\documentclass[12pt]{amsart}
\usepackage[dvips]{epsfig}
\usepackage{amscd}
\usepackage{amssymb}
\usepackage{amsthm}
\usepackage{amsmath}
\usepackage{latexsym}

\usepackage{upref}
\usepackage{hyperref}
\theoremstyle{plain}
\newtheorem{thm}{Theorem}[section]
\theoremstyle{plain}
\newtheorem{lem}[thm]{Lemma}

\newtheorem{cor}[thm]{Corollary}

\theoremstyle{definition}
\newtheorem{defi}{Definition}[section]
\newtheorem*{rem}{Remark}

\newenvironment{Assumptions}
{%
\setcounter{enumi}{0}

\begin{enumerate}}%
{\end{enumerate} }
{%
\setcounter{enumi}{0}

\begin{enumerate}}%
{\end{enumerate} }

\numberwithin{equation}{section} \allowdisplaybreaks

\title[Stochastic Differential Games]{On zero-sum Stochastic Differential Games with Jump-Diffusion driven state: A viscosity solution framework}

\date{}

\author[Imran H. Biswas]{Imran H. Biswas}
\address[Imran H. Biswas]{\newline
 Centre for Applicable Mathematics ,
 Tata Instiute of Fundamental Research,
  P.O.\ Box 6503, GKVK Post Office,
  Bangalore 560065, India}
\email[]{imran@math.tifrbng.res.in}

\subjclass[2000]{45K05, 46S50, 49L20, 49L25, 91A23, 93E20}

\keywords{Stochastic differential games, L\'{e}vy processes, dynamic programming,
 integro-partial differential equation, viscosity solutions.
}

\thanks{The author would like to thank  K.~H.~Karlsen and B. Oksendal for their comments and suggestions. We are also immensely thankful to the reviewers for pointing out several inconsistencies and helping to resolve them in the course of the revision.}

\begin{document}
\begin{abstract}
A zero-sum differential game with controlled jump-diffusion driven state is considered, and studied 
using a combination of dynamic programming and viscosity solution techniques. We prove, under certain conditions, that
the value of the game exists and is the unique viscosity solution of a fully nonlinear
integro-partial differential equation. In addition, we formulate and prove a verification theorem for such games within the
viscosity solution framework for nonlocal equations.
\end{abstract}

\maketitle

\section{Introduction}
In this article we analyze a two-player zero-sum stochastic differential game (SDG henceforth)
where the state is governed by controlled jump-diffusions. For problems related to controlled degenerate
diffusions, viscosity solution setup provides an appropriate framework for analysis. We mention \cite{Fleming:1989cg, Fleming:1993dy,Gozzi:2005bd,Yong:2005fg}
to name a few of the available studies addressing this connection. In one such article \cite{Fleming:1989cg}, the authors used a combination of viscosity solution and dynamic programming techniques to pioneer a comprehensive study of zero-sum SDG. We  extend these results to jump-diffusion driven games and in addition, we use  viscosity solution framework for nonlocal equations to formulate and prove a verification theorem, which is influenced by similar results in \cite{Gozzi:2005bd} related to optimal control problems for diffusions. 

For a fixed positive constant $T$ and $t\in [0,T)$,  let   $\big(\Omega_t,\mathcal{F}_t,P_t,
     \mathcal{F}_{t,\cdot}\big)$ be a filtered probability space satisfying usual hypotheses.  The SDG consists of the following controlled stochastic dynamics,
       defined on $\big(\Omega_t,\mathcal{F}_t,P_t,
     \mathcal{F}_{t,\cdot}\big)$,
\begin{align}\label{eq:SDG_dymcs}
 dX(s)
= &b(s, X(s); Y(s), Z(s))ds+\sigma(s,X(s);Y(s),
Z(s))dW(s)\\\nonumber&+\int_{\mathbb{E}}\eta(s, X(s^-);Y(s),
Z(s);w)\tilde{N}(ds,dw)
\end{align} where $s\in (t,T]$; with the initial condition
\begin{align*}
 X(t) =x~~~(\in\mathbb{R}^d),
\end{align*} and the pay-off functional
 \begin{align}\label{Pay-off}
 J(t, x; Y, Z)= E^{t,x}\Big[\int_{t}^{T}f(s,X(s);Y(s),Z(s)) ds + g\big(X(T)\big)\Big].
 \end{align} The $\sigma$-algebra $\mathcal{F}_t$ consists of subsets of $\Omega_t$, $P_t$ is a probability measure on $\big(\Omega_t, \mathcal{F}_t \big)$, and $\mathcal{F}_{t,\cdot}$ is the shorthand notation for a filtration $\big(\mathcal{F}_{t,s}\big)_{t\le s\le T}$. Furthermore,  $\mathbb{E}=\mathbb{R}^m\backslash\{0\}$ for a positive integer $m$ and  $W(s)$ is a $k$-dimensional Brownian
  motion on the same probability space. $\sigma$'s are $d\times k$ matrices, $b$'s and $\eta$'s are $\mathbb{R}^d$
 valued functions. $N(ds, dw)$ is a Poisson random measure on $\mathbb{E}$ with intensity measure
  $\nu(dw)$ and $\tilde{N}(ds,dw)= N(ds, dw)- \nu(dw)ds$; $Y(\cdot)$ and  $Z(\cdot)$ are two predictable control processes with
   values in  $\mathcal{Y}$
  and $\mathcal{Z}$ respectively. The sets $\mathcal{Y}$ and $\mathcal{Z}$ are two compact metric spaces, respectively representing the control sets of the two players $ I$ and $II$. $E^{t,x}[\cdot]$ means
   the expected value
 of the quantity inside the brackets, and $t,x$ at the superscript signifies that the state process $X(s)$ starts at
 time $t$ from the point $x$. 
The precise assumptions on $\sigma, b, \eta, f, g$ will be stated later but roughly
 speaking, these are Lipschitz continuous in the state-variable and the possibly singular (at origin) Radon measure
 $\nu$, the so-called L\'{e}vy measure, satisfies usual growth restriction. In the zero-sum scenario we
    conventionalize player $I$ to be the minimizing player and player $II$ to be the maximizing player.
     Before the players could start playing the game additional set of rules have to be specified, which we
     describe in the next section.
     Given that the players agree on this set of rules, a value of the game could then be defined.

  In the deterministic case $(\sigma=\eta = 0)$, it is shown 
    under Isaacs condition that the value
      of the game exists and it is the unique viscosity solution of the underlying Bellman-Isaacs equation ( see \cite{Barron;1987df,Evans:1984fv,Souganidis:1985sw}).
       In the
      stochastic case with $\eta = 0$ (pure diffusion), the study of
       SDGs prior to \cite{Fleming:1989cg} was mainly restricted to problems where state is governed by non-degenerate
       diffusions and value functions were linked to classical solutions of second order equations (cf.
        \cite{Bensossan;1974gw,Freidman:1974ga}).
       In \cite{Fleming:1989cg},\linebreak Fleming $\&$ Souganidis adressed the degenerate case and proved under
       similar Isaacs condition that the value of the game exists and is the unique viscosity solution of the
       underlying Isaacs equation which is fully nonlinear, second order, and possibly degenerate.

In recent years, the use of jump-diffusions to realistically model price dynamics in a financial market is becoming increasingly popular
(see \cite{Cont:2004gk} and references therein). In such market models, the activities of market participants could often be formulated as differential games, where the states follow controlled jump-diffusions . We refer to \cite{Oksendal:2007gf, Oksendal:2008} for more in this direction.  
These recent developments mainly focus on proving verification theorems in terms smooth solutions of underlying Bellman-Isaacs equations. With this attempt here, we extend the methodology of \cite{Fleming:1989cg} to look beyond the smooth
         solution setup and provide a rigorous and robust analysis for zero-sum SDGs related to controlled jump-diffusions.

From an intuitive point of view, Bellman-Isaacs equation for the SDG \eqref{eq:SDG_dymcs}-\eqref{Pay-off}
     is a fully nonlinear integro-partial differential equation of the type
         \begin{align}
         \label{eq:bellman_eq_1} u_t+ F(t,x,Du(t,x), D^2u(t,x),u(t,\cdot) )=0\quad \text{in}~[0,T) \times\mathbb{R}^d,
         \end{align}along with the terminal condition
         \begin{align}
           \label{eq:term_cond} u(T,x)= g(x)\quad x\in\mathbb{R}^d.
         \end{align}The term $u(t, \cdot)$ is of special importance to the present article as it represents the nonlocal-ness of the equation, which results directly from the jumps in the dynamics \eqref{eq:SDG_dymcs}. In our context, the equation \eqref{eq:bellman_eq_1} would primarily assume the two following forms
          \begin{align}
       \label{eq:upp-IPDE}  u_t+ H^-(t,x,Du, D^2u,u(t,\cdot) )=0 \quad \text{in}~[0,T)\times\mathbb{R}^d,\\
\nonumber\\
        \label{eq:low-IPDE} u_t+H^+(t,x, Du, D^2u,u(t,\cdot)) =0 \quad \text{in}~[0,T)\times\mathbb{R}^d,
       \end{align} where,
         for $(q,x,t, A)\in
\mathbb{R}^d\times\mathbb{R}^d\times[0,T]\times\mathbb{S}^d$ and a
smooth function $\varphi$,
\begin{align*}
H^-(t,x,q,A,\varphi(t,\cdot))=&\sup_{z\in\mathcal{Z}}\inf_{y\in\mathcal{Y}}\Big[
\mathcal{L}(t,x, q, A; y, z)+\mathcal{J}(t,x;y,z)\varphi\Big]\\
H^+(t,x,q,A,\varphi(t,\cdot)):=&\inf_{y\in\mathcal{Y}}\sup_{z\in\mathcal{Z}}\Big[\mathcal{L}(s,x, q, A; y, z)+\mathcal{J}(t,x;y,z)\varphi\Big]
\end{align*}and
\begin{align*}
\mathcal{L}(t,x,q,A; y,z):=&\text{Tr}(a(t,x;y,z).A)+b(t,x;y,z).p+f(t,x;y,z),\\
  \mathcal{J}(t,x;y,z)\varphi := & \int_{\mathbb{E}}\Big(\varphi(t,x+\eta(t,x,;y,z;w))-\varphi(t,x)\\
  &\hspace{2cm}-\eta(t,x;y,z;w).D\varphi(t,x)
  \Big)\nu(dw)
\end{align*} with $a =\frac 12\sigma \sigma^T$ and $\mathbb{S}^d$ is the set of all symmetric $d\times d$ matrices.

The fully nonlinear integro-PDEs of type \eqref{eq:upp-IPDE} or \eqref{eq:low-IPDE}
     are called degenerate for the following reasons. The matrices $a$'s are assumed to be
merely nonnegative definite and may vanish at some points.
Similarly, the jump vectors $\eta$'s may as well vanish.
Consequently, there are no regularizing effects in these equations coming
from the second order operator (``Laplacian smoothing'') or from the
integral operator (``fractional Laplacian smoothing''). As a result, the equations \eqref{eq:bellman_eq_1}-\eqref{eq:term_cond} will in general not have classical solutions, and
a suitable notion of viscosity solutions is needed. In the past few years, there have been some
efforts to extend the theory of viscosity solution to the
integro-partial differential equations \cite{Alvarez:1996lq,Barles:1997xj,BCI:P07,BI:P07,Jakobsen:2005jy,Jakobsen:2006aa}. This theory is not as developed as its (pure) PDE counterpart, but the available results suffice to ensure the existence,
uniqueness, comparison principles, and some regularity estimates. Next, we mention a few points describing the technical differences in our problem with the existing literature.

In \cite{Fleming:1989cg} the game problem is defined on the Wiener space $C_0\big([t,T]:\mathbb{R}^d\big)$ and
 the structural richness of this space plays a very crucial role in the analysis.
When the stochastic evolutions are driven by L\'{e}vy processes, the underlying sample space is required to reflect that. In our view, the so-called Wiener-Poisson space would be a proper choice for the underlying probability space space. Part of the subtlety for our analysis of the SDG lies in justifying necessary technical assertions related to the sample space.

  Classically, verification theorems (see \cite{Fleming:1993dy} for $\nu =0$ and
     \cite{Oksendal:2005,Oksendal:2007gf} for jump-diffusions) are formulated in terms
     of smooth solutions of the underlying Bellman-Isaacs equation and set up criterion for a set of
      controls for the players to be optimal. But, as has already been pointed out, the Bellman-Isaacs
      equation does not have classical solutions in general, solutions have to be interpreted in the viscosity sense. Therefore, a verification criterion in terms of viscosity
       solutions will have wider applicability. A verification theorem in the framework viscosity solutions for
       first order Bellman equation first appeared in \cite{Zhou:1993yn}. In the context of pure diffusions, a similar
        result appeared in \cite{Zhou:1997cn} for optimal control problems, but with some technical inconsistencies,
        which was later corrected in \cite{Gozzi:2005bd,Gozzi:2010bd}. We follow the ideas from \cite{Gozzi:2005bd,Zhou:1997cn}, and
         formulate a nonlocal version of this result  for SDGs. Even when the jumps are absent, we point out that
          such a verification theorem is new for differential games. We also mention that discontinuities of the sample paths and non-locality
         of the Isaacs equation makes the problem more involved, and new techniques are employed to overcome
          the added difficulties.

          The rest of the paper is organized as follows: in Section 2 we state the full set of assumptions,
        relevant technical details and state the main results. Sections 3 $\&$ 4 respectively contains the proof of dynamic programming principle and verification theorem.

\section{Technical framework and the statements of the main results} 
We use
the notations $Q_T$ and $\bar{Q}_T$ respectively
for $[0,T)\times\mathbb{R}^d$ and $[0,T]\times \mathbb{R}^d$.
For various constants  depending on the data we
mainly use $N,K,C$ with/without subscripts. For a bounded Lipschitz continuous function $h(x)$
defined on $\mathbb{R}^d$, its Lipschitz norm $|h|_1$ is defined as
\begin{align*}
   |h|_1 := \sup_{x\in\mathbb{R}^d} |h(x)|+\sup_{x,y\in
   \mathbb{R}^d}\frac{|h(x)-h(y)|}{|x-y|}.
\end{align*}We denote the space of all $h$ so that $|h|_1< \infty$
by $W^{1,\infty}(\mathbb{R}^d)$ or sometimes only by $W^{1,\infty}$. We also define
\begin{align*}
   C_b^{\frac{1}{2},1}(\bar{Q}_T):= \big\{h(t,x):\sup_{(t,x)\in \bar{Q}_T}|h(t,x)|
   +\sup_{(t,x),(s,y)\in \bar{Q}_T}\frac{|h(t,x)-h(s,y)|}{|t-s|^{\frac
   12}+|x-y|}<\infty\big\}.
\end{align*} 
Furthermore, $|h(t,\cdot)|_1$ simply stands for $|\cdot|_1$  norm
of $h(t,x)$ as a function of $x$ alone with $t$ being fixed. Let $C^{1,2}(Q_T)$
be the space of `once in time' and `twice in space' continuously differentiable functions.
We denote the set of all upper and lower semicontinuous functions on $\bar{Q}_T$
respectively by $USC(\bar{Q}_T)$ and $LSC(\bar{Q}_T)$.
A subscript would mean polynomial growth at infinity, therefore
the spaces $USC_p(\bar{Q}_T) $, $LSC_p(\bar{Q}_T), C_p^{1,2}(Q_T)$
contain the functions $h$ respectively from
$USC(\bar{Q}_T), LSC(\bar{Q}_T)$, $ C^{1,2}(Q_T)$
satisfying the growth condition
\begin{align*}
     |h(x)|\le C(1+|x|^p)\quad \text{for all}
     \quad x\in\mathbb{R}^d~(\text{uniformly in} ~t~\text{if}~h~\text{depends on}~ t ).
\end{align*}
We identify the spaces $USC_0(\bar{Q}_T)$ and $LSC_0(\bar{Q}_T)$ respectively
with $USC_b(\bar{Q}_T)$ and $LSC_b(\bar{Q}_T)$; the subscript `$b$'
signifies boundedness.

 Now we list the precise set of assumptions.
 \begin{Assumptions}
  \item\label{A1} The spaces $\mathcal{Y}$~ and $\mathcal{Z}$ are compact metric spaces; the functions $\sigma, f, b$ and $\eta$ are
   continuous both on $\mathcal{Y}$ ~and $\mathcal{Z}$, uniformly with respect to $(t,x)\in [0,T]\times\mathbb{R}^d$
   and additionally with respect to $ w \in\mathbb{R}^m$ for $\eta$.
   \item\label{A2} $f,b,\sigma, \eta$ are bounded and continuous with respect to $t$ (and $w$ for $\eta$), uniformly in other entries, and there exists a positive constant $K$ such that
       \begin{align*}
        \big(|f|_1+|b|_1+|\sigma|_1\big)(t,\cdot;y,z)+|g|_1\le K,
       \end{align*}uniformly in $(t; y,z; w)$ and
       $$|\eta(t,\cdot; y,z;w)|+|\eta(t,\cdot; y,z;w)|_1\le K\min(|w|,1).$$
   \item\label{A3} In concurrence with \ref{A2}, in this case the L\'{e}vy measure $\nu$ is a positive Radon measure on $\mathbb{E}$ and satisfies 
   \begin{align}
     \int_{\mathbb{E}}\min(|w|^2, 1)\nu(dw) < \infty.
    \end{align}
   \end{Assumptions}
\begin{rem} The assumptions \ref{A1}-\ref{A3} are natural except for the boundedness constraint. However, it is possible to allow certain growth properties and the results of this paper are still valid in a properly modified form. The jump vectors $\eta$ can also enjoy some polynomial growth at infinity in $w$, in which case the L\'{e}vy measure has to have appropriate decay property at infinity.
   \end{rem}
\subsection{Viscosity Solutions for Integro-PDEs} The notion of viscosity solution for nonlocal equations (  such as \eqref{eq:bellman_eq_1}) could be defined in various ways (e.g.  \cite{Arisawa:208za,Jakobsen:2005jy}), but it is not very hard to establish the equivalence. We use the following definition from \cite{Jakobsen:2005jy}.
\begin{defi}\label{defi:visc}For $F=H^+ $ or $ H^-$, 
$v\in USC_p(\bar{Q}_T)(v\in LSC_p(\bar{Q}_T))$ is a viscosity {\em subsolution (supersolution)} of
 \eqref{eq:bellman_eq_1} if for every $(t,x)\in Q_T$ and $\phi\in C_p^{1,2}(Q_T)$ such that $(t,x)$ is a global maximum
 (global minimum) of $v-\phi$,
\[\phi_t +F(t,x,D\phi,D^2\phi,\phi(t,\cdot) )\ge 0 (\le 0).\]
We say that $v$ is a \emph{viscosity solution} of \eqref{eq:bellman_eq_1} if $v$
is both a sub- and supersolution of \eqref{eq:bellman_eq_1}.
\end{defi} The following wellposedness theorem holds, a proof of which can be found in \cite{Jakobsen:2005jy}.
 \begin{thm}\label{thm:wellposed}
 Assume \ref{A1}, \ref{A2} and \ref{A3}. Then, for $F = H^+$ or $F=H^-$, there exists unique viscosity
 solution $u$ to the  terminal value problem
 \eqref{eq:bellman_eq_1}-\eqref{eq:term_cond} and a
 constant $N$ depending only on $d, K, T$ such that
\begin{align}
\label{eq:wellposed}
|v|_{\frac 12, 1}\le N.
\end{align}
Furthermore, a comparison principle holds: If $u$ and $\bar u$ are
bounded sub- and supersolutions of \eqref{eq:bellman_eq_1}-\eqref{eq:term_cond} with $F= H^+$ or $H^-$ and $u(T,\cdot)\le \bar{u}(T,\cdot)$, then $u\leq \bar u$ in $\bar Q_T$.
\end{thm}
    The case $H^+ = H^-$ is of special interest to the present context. We formally say that the {\em Isaacs condition} is satisfied if, for all $(q,x,t, A)\in
\mathbb{R}^d\times\mathbb{R}^d\times[0,T]\times\mathbb{S}^d$ and for
 every smooth function $\varphi$,
   \begin{align}\label{eq:isac_cond}
   H^+(t,x,q,A,\varphi(t,\cdot))= H^-(t,x,q,A,\varphi(t,\cdot)).
   \end{align}

\subsection{The Canonical Sample Space}
  The structural properties of the underlying probability space play an important role in dealing with the technical subtleties involved in the game problem. Contrary to \cite{Fleming:1989cg}, there are jumps in the controlled evolution and we find it convenient to follow \cite{Buckdahn;2008ht,ishikawa;2006ts} and work in a canonical Wiener-Poisson space which is described as follows. 

 For a positive constant $T$ and $0\le s< t\le T$, let $\Omega_{s,t}^1$ be the standard Wiener space i.e. the set of all functions from $[s,t]$ to $\mathbb{R}^d$ starting from $0$ and topologized by the sup-norm. We denote the corresponding Borel $\sigma$-algebra by $\mathcal{B}_1^0$ and let $P_{s,t}^1$ be the Wiener measure on $\big(\Omega_{s,t}^1, \mathcal{B}_1^0\big)$. 

   In addition, upon denoting $\mathrm{Q}_{s,t}^* = [s,t]\times \big(\mathbb{R}^m\backslash\{0\}\big)$, let $\Omega_{s,t}^2$ be the set of all $\mathbb{N}\cup\{\infty\}$-valued measures on $(\mathrm{Q}_{s,t}^*, \mathcal{B}(\mathrm{Q}_{s,t}^*))$ where $\mathcal{B}(\mathrm{Q}_{s,t}^*)$ is the usual Borel $\sigma$-algebra
of $\mathrm{Q}_{s,t}^*$. We denote $\mathcal{B}^0_2$ to be the smallest $\sigma$-algebra over $\Omega_{s,t}^2$ so that the mappings $q\in \Omega_{s,t}^2\mapsto q(A)\in \mathbb{N}\cup\{\infty\}$ are measurable for all $A\in \mathcal{B}(\mathrm{Q}_{s,t}^*)$. Let the co-ordinate random measure $N_{s,t}$ be defined as $N_{s,t}(q, A) = q(A)$ for all  $q\in \Omega_{s,t}^2, A\in \mathcal{B}(\mathrm{Q}_{s,t}^*)$ and denote $P_{s,t}^2$ to be the probability measure on $(\Omega_{s,t}^2,\mathcal{B}^0_2 )$ under which $N_{s,t}$ is a Poisson random measure with L\'evy measure $\nu$ satisfying \ref{A3}. 

 Next, for every $0\le s< t\le T$, we define $\Omega_{s,t} \equiv \Omega_{s,t}^1\times \Omega_{s,t}^2, P_{s,t} \equiv P_{s,t}^1\otimes P_{s,t}^2 $ and $\mathcal{B}_{s,t} \equiv \overline{\mathcal{B}_1^0\otimes \mathcal{B}_2^0}$ i.e. the completion of $\mathcal{B}_1^0\otimes \mathcal{B}_2^0$ with respect to the probability measure $P_{s,t}$. We will follow the convention that $\Omega_{t,T} \equiv \Omega_t$ and $\mathcal{B}_{t,T}\equiv \mathcal{F}_{t}$.  A generic element of $\Omega_t$ is denoted by $\omega = (\omega_1,\omega_2)$, where $\omega_i\in \Omega_{t,T}^i$ for $i\in\{1,2\}$, and we define the coordinate functions
\begin{align*}
 W^t_s(\omega)= \omega_1(s)\qquad\text{and} \quad\quad N^t(\omega, A) = \omega_2(A)
\end{align*} for all $0\le t\le s\le T, \omega\in \Omega, A\in \mathcal{B}(\mathrm{Q}_{t,T}^*)$. The process $W^t$ is a Brownian motion starting at $t$ and $N^t$ is a Poisson random measure on the probability space  $(\Omega_t, \mathcal{F}_t, P_t)$, and they are independent. 

Also, for $t\in [0,T]$, the filtration $\mathcal{F}_{t,\cdot}= (\mathcal{F}_{t,s})_{ s\in[t,T]} $ is defined as follows:
\begin{align*}
 \hat{\mathcal{F}}_{t,s}\equiv \sigma\{W_r^t, N^t(A): A\in \mathcal{B}(\mathrm{Q}_{t,r}^*), t\le r\le s\}, \quad \text{where}\quad t\le s\le T.
\end{align*} We make $\hat{\mathcal{F}}_{t,\cdot}$ to be right-continuous and denote it by $\mathcal{F}_{t,\cdot}^+$.  Finally, we augment  $\mathcal{F}_{t,\cdot}^+$ by $P_t$-null sets and call it  $\mathcal{F}_{t,\cdot}$. As and when it necessitates, we extend the filtration $\mathcal{F}_{t,\cdot}$ for $s< t$ by choosing $\mathcal{F}_{t,s}$ as the trivial $\sigma$ algebra augmented by $P_t$-null sets. We follow the convention that $\mathcal{F}_{t,T}= \mathcal{F}_t$.  When the terminal time $T$ is replaced by another time point, say $\tau$, the filtration we have just described is denoted by $\mathcal{F}^\tau_{t,\cdot}$ . 

Finally, note that the space $\Omega_{s,t}$ is defined as the product of canonical Wiener space and Poisson space. Therefore, for any $\tau\in(t,T)$, we can identify the probability space $\big(\Omega_t, \mathcal{F}_{t,\cdot}, P_t \big)$ with $\big(\Omega_{t,\tau}\times\Omega_\tau, \mathcal{F}_{t, \cdot}^\tau\otimes \mathcal{F}_{\tau,\cdot}, P_{t, \tau}\otimes P_{\tau}\big)$ by the following bijection $\pi:\Omega_t\rightarrow \Omega_{t,\tau}\times\Omega_\tau$. For a generic element $\omega = (\omega_1, \omega_2)\in \Omega_t = \Omega_{t,T}^1\times \Omega_{t,T}^2$, we define 

\begin{align*}
 \omega^{t,\tau}&=\big(\omega_1|_{[t,\tau]}, \omega_2|_{[t, \tau]} \big)\in \Omega_{t,\tau}\\
 \omega^{\tau, T} & = \big((\omega_1-\omega_1(\tau))|_{[\tau, T]}, \omega_2|_{[\tau, T]} \big)\in \Omega_{\tau,T}\\
\pi(\omega) &= (\omega^{t,\tau},\omega^{\tau, T})
\end{align*} The description of the inverse map $\pi^{-1}$ is also apparent from above. 

\subsection{Rule of the game}
The player $I$ controls $Y(\cdot)$ and player $II$ chooses $Z(\cdot)$ respectively to minimize and maximize $J$. At any given time $s\in(t,T)$, the players know $\big(X,Y,Z\big)(r)$ for $r< s$ and instantaneous switching at $s$ is possible. Therefore the player who acts first at $s$ is apparently at disadvantage. We follow \cite{Fleming:1989cg} to tackle this problem and formalize two approximate games, namely the upper and lower game. For the lower-game, the player $I$ is allowed to know $Z(s)$ before choosing $Y(s)$, and for the upper-game, player $II$ has the upper hand of knowing $Y(s)$ before choosing $Z(s)$. Next step would be to define upper and lower value of the game. Before that, we need to define the concepts of admissible controls and strategies for the players. 

\begin{defi}[admissible control]
 An admissible control process $Y(\cdot)$(resp. $Z(\cdot)$) for
 player $I$(resp. player $II$) on $[t,T]$ is a  $\mathcal{Y}$(resp. $\mathcal{Z}$)-valued process which
 is $\mathcal{F}_{t,\cdot}$- {\em predictable}. The set of all admissible
 controls for player $I$(resp. $II$) on $[t,T]$ is denoted by $M(t)$(resp.
 N(t)).  We say the controls $Y,\tilde{Y}\in M(t)$ are the same on $[t,s]$ and we write $Y\thickapprox
   \tilde{Y}$ on $[t,s]$ if $P_t\big(Y(r)=
   \tilde{Y}(r)~\text{for a.e.}~ r\in [t,s] \big)= 1 $.  A similar convention is followed for members of $N(t)$. 
 \end{defi}
 Finally, if $Y\in M(t)$, then for every $s\in [t,T]$ there exists $Y^s(\cdot):[t,s]\times \Omega_{t,s}\rightarrow \mathcal{Y}$ such that $Y(r,\omega)= Y ^s(r,\omega^s)$ where $r\in [t,s], \omega\in \Omega_t$, $\omega^s= \omega|_{[t,s]}$ and $Y^{s}(\cdot)$ is a $\mathcal{F}^{s}_{t,\cdot}$-predictable process. 
 
\begin{rem}
   In comparison with \cite{Fleming:1989cg}, we require the control processes to be predictable so that the integrand  $\eta(\cdot)$ in \eqref{eq:SDG_dymcs} is also predictable.  Also, as pointed out by one reviewer, any $\mathcal{F}_{t,\cdot}$-predictable control process $Y$ will have an $\mathcal{F}_{t,\cdot}^+$-predictable version which will have the representation as described above. 
\end{rem}

\begin{defi}[admissible strategy]
   An admissible strategy $\alpha$ (resp $\beta$) for player $I$ (resp. $II$)
   is a mapping $\alpha: N(t)\rightarrow M(t)$ (resp. $\beta:M(t)\rightarrow N(t)
   $) such that if $Y(\cdot)\thickapprox
   \tilde{Y}$(resp. $Z\thickapprox
   \tilde{Z}$) on $[t,s]$, then $\alpha[Y]\thickapprox
   \alpha[\tilde{Y}]$ ( resp. $\beta[Z]\thickapprox
   \beta[\tilde{Z}]$ ) on $[t,s]$ for every $s\in [t, T]$.  The set of
   admissible strategy for player $I$ (resp. $II$) on $[t,T]$ is
   denoted by $\Gamma[t]$ (resp. $\Delta(t)$).
  \end{defi}

\begin{defi}[]({\em Value functions})
 \begin{itemize}
 \item[i.)] The lower value of the SDG \eqref{eq:SDG_dymcs}-\eqref{Pay-off}
 with initial data $(t,x)$ is given by
 \begin{align}\label{upperval}
 U(t,x):= \inf_{\alpha \in \Gamma(t)}\sup_{Z\in N(t)}J\big(t,x; \alpha[Z],Z\big).
 \end{align}
 \item[ii.)] The upper value of the game is defined as follows,
\begin{align}\label{lowerval}
 V(t,x):= \sup_{\beta \in \Delta(t)}\inf_{Y\in M(t)}J\big(t,x; Y,\beta[Y]\big).
 \end{align}
 \end{itemize}
 \end{defi} We say that our game has a value in the sense of Elliot $\&$ Kalton \cite{Elliot:1972ns} if $V(t,x)=U(t,x)$ and call this common value to be the value of the game. The upper and lower values satisfy the following dynamic programming principle, a detailed proof of which is given in Section 3.
\begin{thm}\label{dpp}
   Let \ref{A1},\ref{A2},\ref{A3} hold and  $t,\tau\in [0,T]$ be such that $t< \tau$.
   For every $x\in
   \mathbb{R}^d$, we have
   \begin{align}\label{dpp-low}
 V(t,x) =\sup_{\beta\in \Delta(t)}\inf_{Y\in
    M(t)}E^{t,x}\Big\{\int_{t}^{\tau}f(s, X(s); Y(s), \beta[Y](s))ds+V(\tau,X_\tau)\Big\}
   \end{align}
   where $X(\cdot)$ is the solution of \eqref{eq:SDG_dymcs} with
   $Z(\cdot)=\beta[Y](\cdot)$ for $Y(\cdot)\in M(t)$, and
\begin{align}\label{dpp-up}
 U(t,x) =\inf_{\alpha\in \Gamma(t)}\sup_{Z\in
    N(t)}E^{t,x}\Big\{\int_{t}^{\tau}f(s, X(s); \alpha[Z](s), Z(s))ds+U(\tau,X_\tau)\Big\}
   \end{align}
   where $X(\cdot)$ is the solution of \eqref{eq:SDG_dymcs} with
   $Y(\cdot)=\alpha[Z](\cdot)$ for $Z(\cdot)\in N(t)$.
  \end{thm}

 The proof of Theorem \ref{dpp} is not based on probabilistic techniques alone. Following \cite{Fleming:1989cg} we first prove a part of the following viscosity solution connection, and then use it to prove Theorem \ref{dpp}. The statement of the theorem reads is given below, a proof of which is given in Section 3.

 \begin{thm}\label{thm:visc_connection}Let \ref{A1},\ref{A2} and \ref{A3} hold. Then the upper-value $V$ and the lower-value $U$ of the game \eqref{eq:SDG_dymcs}-\eqref{Pay-off} are respectively the unique viscosity solutions of \eqref{eq:upp-IPDE}-\eqref{eq:term_cond} and \eqref{eq:low-IPDE}-\eqref{eq:term_cond}.

 \end{thm}
\begin{rem}
    It is now obvious that if the Isaacs condition \eqref{eq:isac_cond} holds, then uniqueness of viscosity solution of IPDEs forces the upper and lower value of the game to coincide in view of Theorem \ref{thm:visc_connection}. This ensures existence of the value, in the sense of Elliot $\&$ Kalton \cite{Elliot:1972ns}, of our SDG. 
\end{rem}

\subsection{Stochastic Verification Theorem}  Before the verification theorem could be formulated, some further technical preparations are needed.  Given a probability space  $(\Omega,\mathcal{F}, P)$ with a filtration
      $\mathcal{F}_{a,\cdot}=\big\{\mathcal{F}_{a,s}: a\le s\le b\big\}$ and  a separable Banach space $\mathbb{B}$ with norm $|\cdot|_\mathbb{B}$ and
      $1\le p<\infty$, the space $L^p_{\mathcal{F}_{a,\cdot}}(a,b; \mathbb{B})$ is defined as follows:
         \begin{align*}
            L^p_{\mathcal{F}_{a,\cdot}}(a,b; \mathbb{B})&=\Big\{\phi(s,\omega), a\le s\le b| \phi(s,\cdot) ~\text{is an}~
         \mathcal{F}_{a,s}-\text{adapted},\mathbb{B}-\text{valued}\\&~\quad\quad\text{measurable process on}
         ~[a,b]~\text{and}~E\big(\int_a^b|\phi(s,\omega)|_\mathbb{B}^pds\big)<\infty \Big\}.
         \end{align*}
\begin{defi}\label{def:parabolic-jets}
  We say that a triplet $(p,q,Q)\in \mathbb{R}\times\mathbb{R}^d\times\mathbb{S}^d$ is in $D_{s+,x}^{1,2+} v(s,x)$, the
  second order one-sided parabolic superdifferential of $v$ at $(s,x)$, if for all $y\in\mathbb{R}^d$ and $r\ge s$,
       \begin{align*}
         v(r,y) \le v(s,x) + p(s-r)+\langle q, y-x\rangle+\langle Q(y-x), y-x\rangle+ o(|s-r|+|y-x|^2).
       \end{align*}The second order one sided parabolic subdifferential of $v$ at $(s,x)$; $D_{s+,x}^{1,2-} v(s,x)$ is
       defined by reversing the  above inequality i.e. $D_{s+,x}^{1,2-} v(s,x)= -D_{s+,x}^{1,2+} (-v(s,x))$
\end{defi} We state the following lemma, well-known in context of viscosity solution theory, characterizing
super and subdifferentials.
   \begin{lem}\label{lem:superdiff-charac}
   Let $v\in USC([0,T]\times\mathbb{R}^d)$ and $(s_0,x_0)\in[0,T)\times\mathbb{R}^d$. Then $(p,q,Q)\in D_{s_0+,x_0}^{1,2+} v(s_0,x_0)$ iff
    there exists a function $\varphi\in C^{1,2}([0,T]\times\mathbb{R}^d)$ such that $v-\varphi $ has a strict
    global maximum at $(s_0,x_0)$ relative to the set $(s,x)$ such that $s\ge s_0$ and
    \begin{align}\label{eq:eqv-jet}
    [\varphi,\varphi_t, D\varphi,D^2\varphi](s_0,x_0)= [v(s_0,x_0), p, q, Q].
    \end{align} Furthermore, if $v(t,x)$ has polynomial growth, i.e. if
    \begin{align}
    \label{eq:growth} |v(s, x)|\le C(1+|x|^k) \quad\text{for some}~k\ge 1\quad\text{and}~(s,x)\in[0,T]\times \mathbb{R}^d,
    \end{align}then $\varphi$ can be chosen so that $\varphi,\varphi_t,D\varphi, D^2\varphi$ all satisfy the same
    growth condition \eqref{eq:growth}, possibly with a different constant in place of $C$.
   \end{lem} A detailed proof of Lemma \ref{lem:superdiff-charac} could be found in \cite{Yong:2005fg},
    (Lemma 5.4, Chapter 4. to be precise). The following equivalent characterization of Definition \ref{defi:visc} holds. The proof is similar to the local case and follows by combining Lemma \ref{lem:superdiff-charac} and the reasoning \cite[p.2012]{Gozzi:2005bd}. For the sake of completeness of our presentaion, we sketch the proof Lemma \ref{lem:eqv-visc} in the Appendix. 
     \begin{lem}\label{lem:eqv-visc}
       A $v\in USC(\bar{Q}_T)$ is a subsolution of \eqref{eq:bellman_eq_1} with $F= H^+$ or $H^-$, if and only if, for all
       $(t,x)\in Q_T$ and $(p, q, Q)\in D_{t,x}^{1,2+} v(t,x)$
       \begin{align*}
       p+F(t,x,q,Q,\varphi(t,\cdot))\ge 0,
       \end{align*} where $\varphi= \varphi(p,q,Q)$ given by Lemma \ref{lem:superdiff-charac}
       satisfying \eqref{eq:eqv-jet} at $(t,x)$.
     \end{lem} With slight abuse of notation, for the rest of this section we denote the space of all functions $v\in C^{1,2}(Q_T)$
     with $v, Dv$ and $D^2v(t,x)$ satisfying \eqref{eq:growth} by $C_k^{1,2}(Q_T)$. In fact, $C_k^{1,2}(Q_T)$ is a separable
     Banach space with respect to the usual weighted norm. We are now ready to phrase the
     verification theorem.

\begin{thm}[Verification Theorem]\label{thm:verification}
       Assume \ref{A1},\ref{A2}, \ref{A3} and the condition \eqref{eq:isac_cond} holds. Let $u, v\in C_1(\bar{Q}_T)$
       be respectively a sub and supersolution of \eqref{eq:upp-IPDE} satisfying \eqref{eq:term_cond}. Fix any
       $(t,x)\in Q_T$. Let $(Y^*, Z^*)\in M(t)\times N(t)$ be an admissible control pair for the SDG
       \eqref{eq:SDG_dymcs}-\eqref{Pay-off} starting at $(t,x)$ and $X^*(\cdot)$ be the corresponding solution of
       \eqref{eq:SDG_dymcs}. Suppose that there exist c\`{a}dl\`{a}g processes $(p^i, q^i, Q^i; \Phi^i)\in L_{\mathcal{F}_{t,\cdot}}^2(t,T;\mathbb{R})\times
        L_{\mathcal{F}_{t,\cdot}}^2(t,T;\mathbb{R}^d)\times L_{\mathcal{F}_{t,\cdot}}^2(t,T;\mathbb{S}^d)\times
        L_{\mathcal{F}_{t,\cdot}}^2(t,T;C_1^{1,2}(Q_T))$ for $i\in\{1,2\}$ with $\Phi^i$ progressively measurable, for $i\in\{1,2\}$, such that
        \begin{itemize}
        \item[a.)]for a.e.
        $s\in[t,T]$ and $i=1$
        \begin{align}
        \label{eq:diff-sup}(p^1(s), q^1(s), Q^1(s))\in D^{1,2 +}_{s+}u(s, X^*(s))~~\text{and}
        \end{align}$u-\Phi^1_s$ has a global maximum at $(s,X^*(s))$ with $\Phi_s^1(s,X^*(s))$ $= u(s,X^*(s))$ $P_t$-a.s. and
        \begin{align}
          \label{eq:diff-sup-integral}E^{t,x}\Big\{\int_t^T[p^1(s)&+\mathcal{L}(s, X^*(s), p^1(s),q^1(s), Q^1(s); Y^*(s),Z^*(s))
          \\\nonumber&\quad+\mathcal{J}(s,X^*(s); Y^*(s),Z^*(s))\Phi^1_s(s,X^*(s))]ds\Big\}\le 0.
        \end{align}
        \item[b.)]  For a.e.
        $s\in[t,T]$ and $i=2$
        \begin{align}
        \label{eq:sub-diff}(p^2(s), q^2(s), Q^2(s))\in D^{1,2 -}_{s+}v(s, X^*(s))~~\text{and}
        \end{align}$v-\Phi^2_s$ has a global minimum at $(s,X^*(s))$ with $\Phi_s^2(s,X^*(s))$ $= v(s,X^*(s))$ $P_t$-a.s. and
        \begin{align}
          \label{eq:sub-diff-integral}E^{t,x}\Big\{\int_t^T[p^2(s)&+\mathcal{L}(s, X^*(s), p^2(s),q^2(s), Q^2(s); Y^*(s),Z^*(s))
          \\\nonumber&\quad+\mathcal{J}(s,X^*(s); Y^*(s),Z^*(s))\Phi^2_s(s,X^*(s))]ds\Big\}\ge 0.
        \end{align}
        \end{itemize}
     Then $(Y^*, Z^*)$ is an `optimal' control-pair for the SDG in the sense that
     \begin{align}
       \label{eq:finalconc}U(t,x)=V(t,x)=J(t,x;Y^*,Z^*).
     \end{align}
     \end{thm}
     \begin{rem} The redundancy in the above statement is apparent for following reasons. Given the assumptions \ref{A1},\ref{A2} and\ref{A3},  if $\Phi^i$ is progressively 
     measurable, the integrability condition on $(p^i, q^i, Q^i)$ is automatically satisfied and one can replace them by derivatives of $\Phi^i_\cdot$ at $(\cdot, X^*(\cdot))$. Therefore, it is possible to equivalently state the theorem without introducing $(p^i, q^i, Q^i)$. However, we adopt this particular format on purpose.  As has already been mentioned, even for pure diffusions, so far no verification theorem has been formulated for SDG using viscosity solution framework. In such a scenario, an appropriate formulation would be to drop  $\Phi^i$ and leave the statement in terms of semijets only.
     \end{rem}

\section{Proof of dynamic programming principle}
    We start this section with the following observation as an immediate consequence of the definition of admissible controls and the remark following the definition. We state this as a lemma for later reference.
    \begin{lem}\label{lem:control-breaup}
       Let $0\le\bar{t}< t< T$ and  $Y(\cdot)\in M(\bar{t})$. Then, for $P_{\bar{t},t}$-a.e. $\omega^1\in\Omega_{\bar{t},t}$, the map $Y(\omega^1): [t,T]\times \Omega_t\rightarrow \mathcal{Y}$ defined by
       $Y(\omega^1)(r,\omega^2):=Y(r,\pi^{-1}(\omega^1,\omega^2))$ is a version of an $\mathcal{F}_{t,\cdot}$-predictable process. A similar assertion holds for members of $N(\bar{t})$.
    \end{lem}

As has been pointed out in Section 2, on the probability space $\big(\Omega_t, \mathcal{F}, \mathcal{F}_{t,\cdot}, P_t\big)$, there exists unique solution $X_{t,x}(\cdot)$ of 
the SDE \eqref{eq:SDG_dymcs} for any $4$-tuple $(t,x,Y, Z)\in[0,T)\times\mathbb{R}^d\times M(t)\times N(t)$  i.e.
\begin{align}
    \label{eq:int_represent} X_{t,x}(s) &= X_{t,x}(\tau)+\int_{\tau}^sb(r,X_{t,x}(r);\gamma(r))dr+\int_{\tau}^s
       \sigma(r,X_{t,x}(r);\gamma(r))dW^t(r)\\&
       \nonumber \qquad\qquad +\int_{\tau}^s\int_{\mathbb{E}}
       \eta(r,X_{t,x}(r^-);\gamma(r); w)d\tilde{N}^t(dr,dw),
\end{align} where $\tau\le s\le T$ and $\gamma(\cdot)$ 
is a shorthand for the pair $(Y(\cdot), Z(\cdot))$. For any $\tau\in (t,T]$ and $\omega^{t,\tau}\in\Omega_{t,\tau},\omega^{\tau,T}\in \Omega_\tau$, we define 
 \begin{align*}&\tilde{\gamma}(r,\omega^{t,\tau},\omega^{\tau,T}) \equiv  \gamma(r,\pi^{-1}(\omega^{t,\tau},\omega^{\tau,T}))\\\text{and}\quad &\tilde{X}(s, \omega^{t,\tau},\omega^{\tau,T}) = X_{t,x}(s,\pi^{-1}(\omega^{t,\tau},\omega^{\tau,T})).\end{align*} In addition, for 
$P_{t,\tau}$-a.e. $\omega^{t,\tau}\in \Omega_{t,\tau}$, we wish to have 
\begin{align}
     \label{eq:path_break_up} \tilde{X}(s,\omega^{t,\tau},\cdot) &=X_{t,x}(\tau)
      +\int_{\tau}^sb(r,\tilde{X}(r,\omega^{t,\tau},\cdot );\tilde{\gamma}(r,\omega^{t,\tau},\cdot))dr\\ \nonumber & \quad+\int_\tau^s
      \sigma(r,\tilde{X}(r,\omega^{t,\tau},\cdot );\tilde{\gamma}(r,\omega^{t,\tau},\cdot))dW^\tau(r)\\
      \nonumber &\qquad
      +\int_\tau^s\int_{\mathbb{E}}
      \eta(r,\tilde{X}(r^-,\omega^{t,\tau},\cdot );\tilde{\gamma}(r,\omega^{t,\tau},\cdot);w)d\tilde{N}^{\tau}(dr,dw).
\end{align}
For $\tau\in [t,T]$, it follows straight from the definition that
\begin{align*}
 W^t_s\big(\pi^{-1}(\omega^{t,\tau},\omega^{\tau,T})\big)-W^t_\tau\big(\pi^{-1}(\omega^{t,\tau},\omega^{\tau,T})\big) = \omega^\tau(s).
\end{align*} For the Poisson random measure, it also follows from the definition that
 $$\big(\pi^{-1}(\omega^{t,\tau},\omega^{\tau,T})\big)(A) = \omega^{\tau,T}_2(A)$$
where $A$ is any Borel subset of $[\tau,T]\times \mathbb{R}^m\backslash\{0\}$. Therefore, for $P_{t, \tau}$-a.e  $\omega^{t,\tau}\in\Omega_{t,\tau}$, the processes $W^t_s\big(\pi^{-1}(\omega^{t,\tau},\omega^{\tau,T})\big)-W^t_\tau\big(\pi^{-1}(\omega^{t,\tau},\omega^{\tau,T})\big)$ and the random measure $N^t(\pi^{-1}(\omega^{t,\tau},\omega^{\tau,T}))$ respectively coincides with the canonical Brownian motion $W^\tau_s$ and the canonical Poisson random measure $N^\tau$ on the probability space $(\Omega_\tau, \mathcal{F}_{\tau,\cdot}, P_\tau)$. Therefore, for $P_{t,\tau}$-a.e. $\omega^{t,\tau}\in \Omega_{t,\tau}$, the equality \eqref{eq:path_break_up} holds as a consequence of \eqref{eq:int_represent}. 
We invoke the uniqueness for the SDE \eqref{eq:SDG_dymcs} to conclude that the paths of $\tilde{X} (s, \omega^{t,\tau},\cdot)_{s\in [\tau, T]}$ will coincide with those of the solution of \eqref{eq:SDG_dymcs} on $\Omega_\tau$ with initial condition $(\tau,X_{t,x}(\tau))$ and control pair [as declared in Lemma \ref{lem:control-breaup}] $(Y(\omega^{t,\tau}),Z(\omega^{t,\tau}))(\cdot)$ , for $P_{t,\tau}$-almost all $\omega^{t,\tau}$. Unless otherwise mentioned, we use the same notation $X_{t,x}$ when it is considered as a process on $(\Omega_\tau, P_\tau)$ and we actually mean the process $\tilde{X}$ on $(\Omega_\tau, P_\tau)$. 
\begin{lem}\label{markovprop}
Let $X_{t,x}(\cdot)$ be the solution of \eqref{eq:SDG_dymcs} for any $4$-tuple $(t,x,Y(\cdot), Z(\cdot) )\in[0,T)\times\mathbb{R}^d\times M(t)\times N(t)$. For any bounded continuous function $\psi$,  and $s\in [\tau,T]$ (deterministic),  it holds that 
\begin{align}
       \label{markovp}
       &E^{t,x}[\psi(X_{t,x}(s),\gamma(s))|\mathcal{F}_{t,\tau}](\pi^{-1}(\omega^{t,\tau},\omega^{\tau, T}))\\
       \nonumber =&E^{\tau,X_{t,x}(\tau)}[\psi(X_{t,x}(s),\tilde{\gamma}(s,\omega^{t,\tau},\omega^{\tau,T}))],\quad
       P_{t,\tau}-a.s. 
\end{align}
\end{lem}
\begin{proof}
  For a bounded and measurable function $\varphi$, we have 
\begin{align}
\label{eq:markov1} &E^{P_t}\big[\varphi(\omega)|\mathcal{F}_{t,\tau}\big](\pi^{-1}(\omega^{t,\tau},\omega^{\tau,T}))\\
 \nonumber = & E^{P_{t,\tau}\otimes P_\tau}\big[\varphi(\pi^{-1}(\omega^{t,\tau},\omega^{\tau,T}))|\mathcal{F}_{t,\tau}^\tau\otimes \mathcal{F}_{\tau,\tau}^0\big](\omega^{t,\tau},\omega^{\tau,T}),
\end{align} where $\mathcal{F}_{\tau,\tau}^0$ is the trivial $\sigma$ algebra on the proability space $ (\Omega_\tau, P_\tau)$. Therefore $E^{P_t}\big[\varphi(\omega)|\mathcal{F}_{t,\tau}\big](\pi^{-1}(\omega^{t,\tau},\omega^{\tau,T}))$ is $\mathcal{F}_{t,\tau}^\tau\otimes \mathcal{F}_{\tau,\tau}^0$ measurable. Thereby applying Fubini's theorem we conclude that, for $P_{t,\tau}$-a.e. $\omega^{t,\tau}\in \Omega_{t,\tau}$, the map $\omega^{\tau, T} \mapsto E^{P_t}\big[\varphi(\omega)|\mathcal{F}_{t,\tau}\big]$\linebreak$(\pi^{-1}(\omega^{t,\tau},\omega^{\tau,T})) $ is $\mathcal{F}_{\tau,\tau}^0$ measurable. In other words,  $E^{P_t}\big[\varphi(\omega)|\mathcal{F}_{t,\tau}\big](\pi^{-1}(\omega^{t,\tau},\cdot))$ is a constant function on $\Omega_\tau$ for $P_{t,\tau}$-a.e. $\omega^{t,\tau}\in \Omega_{t,\tau}$. For $A\in \mathcal{F}_{t,\tau}^\tau$ we can write
\begin{align*}
 &\int_{A}E^{P_t}\big[\varphi |\mathcal{F}_{t,\tau}\big](\pi^{-1}(\omega^{t,\tau},\cdot))dP_{t,\tau}\\
= & \int_{\Omega_\tau}\int_{A}E^{P_t}\big[\varphi |\mathcal{F}_{t,\tau}\big](\pi^{-1}(\omega^{t,\tau},\omega^{\tau, T}))dP_{t,\tau}dP_\tau\\
=& \int_{\pi^{-1}(A\times \Omega_{\tau})}E^{P_t}\big[\varphi(\omega) |\mathcal{F}_{t,\tau}\big]dP_t(\omega) \\
=&  \int_{\pi^{-1}(A\times \Omega_{\tau})} \varphi(\omega)dP_t(\omega)\\
=&\int_A\int_{\Omega_\tau}\varphi(\pi^{-1}(\omega^{t,\tau},\omega^{\tau,T}))dP_{\tau}(\omega^{\tau,T})dP_{t,\tau}(\omega^{t, \tau})\\
=& \int_A E^{P_{\tau}}\big[\varphi(\pi^{-1}(\omega^{t,\tau},\omega^{\tau,T}))\big]dP_{t,\tau}(\omega^{t, \tau}).
\end{align*} Therefore, for $P_{t,\tau}$-a.e. $\omega^{t,\tau}\in \Omega_{t,\tau}$, we have 
\begin{align}
 \label{eq:markov2} E^{P_t}\big[\varphi |\mathcal{F}_{t,\tau}\big](\pi^{-1}(\omega^{t,\tau},\omega^{\tau, T})) =  E^{P_{\tau}}\big[\varphi(\pi^{-1}(\omega^{t,\tau},\omega^{\tau,T}))\big].
\end{align} Now define $\varphi(\omega) =\psi(X_{t,x}(s,\omega),\gamma(s,\omega))$, use \eqref{eq:markov2} and invoke the description of $X_{t,x}(\cdot)$ next to Lemma \ref{lem:control-breaup} to conclude, for $P_{t,\tau}$-a.e. $\omega^{t,\tau}\in \Omega_{t,\tau}$, that 
\begin{align*}
 &\quad E^{t,x}[\psi(X_{t,x}(s),\gamma(s))|\mathcal{F}_{t,\tau}](\pi^{-1}(\omega^{t,\tau},\omega^{\tau, T}))\\&=E^{P_t}[\varphi(\omega)|\mathcal{F}_{t,\tau}](\pi^{-1}(\omega^{t,\tau},\omega^{\tau, T}))\\
     &= E^{P_{\tau}}\big[\varphi(\pi^{-1}(\omega^{t,\tau},\omega^{\tau,T}))\big]\\
     & = E^{P_{\tau}}\big[\psi(X_{t,x}(s,\pi^{-1}(\omega^{t,\tau},\omega^{\tau,T})), \tilde{\gamma}(s,\omega^{t,\tau},\omega^{\tau,T})\big]\\
      & = E^{\tau,X_{t,x}(\tau)}[\psi(X_{t,x}(s),\tilde{\gamma}(s,\omega^{t,\tau},\omega^{\tau,T}))]. 
\end{align*} 

\end{proof}

\begin{lem}\label{lem:lip_holder}Let \ref{A1},\ref{A2},\ref{A3} hold. Then
 \begin{itemize}
   \item[a.)] for every $Y\in M(t),  Z\in N(t), \alpha\in\Gamma(t)$ and  $ \beta\in\Delta(t)$,  the pay-off functionals 
    $J (t, \cdot; Y, \beta[Y] )$ and $J(t,\cdot; \alpha[Z], Z)$ are bounded and  Lipschitz continuous in $x$, uniformly in $t, Y, Z,\alpha, \beta$.

   \item[b.)] The value functions $U$ and $V$ are bounded and Lipschitz continuous in $x$.
 \end{itemize}
 \end{lem}
    This lemma is a consequence of  Lipschitz continuity and boundedness of the
data, moment estimates for the stochastic processes, and Gronwall's
inequality. The details of the proof will be given in the Appendix. 

 In order to prove the DPP, it does not seem to be possible to replicate the same strategy available for proving DPP for deterministic games. There will be serious measurability issues. We follow \cite{Fleming:1989cg} and work around this problem with a restricted class of strategies for both the players which we name as $r$-strategies, keeping in line with \cite{Fleming:1989cg}.
 \begin{defi}
     An $r-strategy$ $\beta$ for player $II$ on $[t,T]$ is an admissible strategy with the following additional property: For every $\bar{t}<t<T$ and $Y(\cdot)\in M(\bar{t})$, the map $(r,\omega)\longmapsto\beta[Y(\omega^{\bar{t},t})](r,\omega^{t,T})$ is $\mathcal{F}_{\bar{t},\cdot}$-predictable,  where
       $Y(\omega^{\bar{t},t})$ is defined in Lemma \ref{lem:control-breaup}. The set of all $r-strategies$ for player $II$ on [t,T] is denoted by $\Delta_1(t)$.
    \end{defi}The $r-strategies$ for player $I$ are similarly defined on $[t,T]$ and the set is denoted by $\Gamma_1(t)$. We restrict the choices of the players only  to the $r$-strategies and define the $r$-upper value and $r$-lower value of the game as follows:
    
    \vspace{.3cm}
      \begin{defi}[$r$-values]~
 \begin{itemize}
 \item[i.)] The $r$-lower value of the SDG \eqref{eq:SDG_dymcs}-\eqref{Pay-off}
 with initial data $(t,x)$ is given by
 \begin{align}\label{eq:r-upperval}
 U_1(t,x):= \inf_{\alpha \in \Gamma_1(t)}\sup_{Z\in N(t)}J\big(t,x; \alpha[Z], Z\big).
 \end{align}
 \item[ii.)] The $r$-upper value of the game is defined as follows,
\begin{align}\label{eq:r-lowerval}
 V_1(t,x):= \sup_{\beta \in \Delta_1(t)}\inf_{Y\in M(t)}J\big(t,x; Y, \beta[Y]\big).
 \end{align}
 \end{itemize}
 \end{defi}
As a corolllary to Lemma \ref{lem:lip_holder}, one can derive  the following regularity properties of the $r$-value functions. 
 \begin{cor}\label{cor:r-value}~
 \begin{itemize}
   \item[a.)]The $r$-value functions $U_1$ and $V_1$ are bounded and Lipschitz continuous in $x$, uniformly in $t$.
    \item[b.)] For every $(t,x)\in [0,T]\times\mathbb{R}^d$
    \begin{align*}
       U_1(t,x)\ge U(t,x)\quad \text{and}~V_1(t,x)\le V(t,x).
    \end{align*}
 \end{itemize}
 \end{cor}
 \begin{proof}
    The proof of part [$ b.)$] is obvious. The argument for part [$a.)$] is exactly the same as Lemma \ref{lem:lip_holder}.
 \end{proof}
 
 The $r$-value functions do not satisfy the equalities in Theorem \ref{dpp} (dynamic programming principle), each of them instead satisfies an inequality. We have the following theorem.

 \begin{thm}\label{thm:r_dpp}
   Let \ref{A1},\ref{A2},\ref{A3} hold and  $t,\tau\in [0,T]$ be such that $t< \tau$.
   For every $x\in
   \mathbb{R}^d$, we have
   \begin{align}\label{eq:dpp-sub}
 V_1(t,x) \ge \sup_{\beta\in \Delta_1(t)}\inf_{Y\in
    M(t)}E^{t,x}\Big\{\int_{t}^{\tau}f(s, X(s); Y(s), \beta[Y](s))ds+V_1(\tau,X(\tau))\Big\}
   \end{align}
   where $X(\cdot)$ is the solution of \eqref{eq:SDG_dymcs} with
   $Z(\cdot)=\beta[Y](\cdot)$ for $Y(\cdot)\in M(t)$,  and
\begin{align}\label{eq:dpp-sup}
 U_1(t,x) \le\inf_{\alpha\in \Gamma_1(t)}\sup_{Z\in
    N(t)}E^{t,x}\Big\{\int_{t}^{\tau}f(s, X(s); \alpha[Z](s), Z(s))ds+U_1(\tau,X(\tau))\Big\}
   \end{align}
   where $X(\cdot)$ is the solution of \eqref{eq:SDG_dymcs} with
   $Y(\cdot)=\alpha[Z](\cdot)$ for $Z(\cdot)\in N(t)$.
  \end{thm} 

\begin{proof}
The proofs of \eqref{eq:dpp-sub} and \eqref{eq:dpp-sup} are similar to one another, we only provide the details for \eqref{eq:dpp-sub}. Fix $(t,x)\in Q_T$ and define
  \begin{align}
  \label{eq:dpp-1}W(t,x) = \sup_{\beta\in \Delta_1(t)}\inf_{Y\in
    M(t)}E^{t,x}\Big\{\int_{t}^{\tau}f(s, X(s); Y(s), \beta[Y](s))ds+V_1(\tau,X_\tau)\Big\}.
    \end{align} For every $\epsilon > 0$, then there exists $\beta^\epsilon\in \Delta_1(t)$ such that
  \begin{align}
 \label{eq:dpp-2} W(t,x) \le E^{t,x}\Big\{\int_{t}^{\tau}f(s, X(s); Y(s), \beta^\epsilon[Y](s))ds+V_1(\tau,X_\tau)\Big\}+\epsilon,
  \end{align} for every $Y(\cdot)\in M(t)$. Recall the definition of $V_1(t,x)$ and argue for every $\xi\in\mathbb{R}^d$ that there exists $\beta^\xi\in \Delta_1(\tau)$ such that
 \begin{align}
   \label{eq:dpp-3}V_1(\tau,\xi)\le J\big(\tau,\xi; Y,\beta^\xi[Y]\big)+\epsilon\quad\text{for every}\quad Y\in M(\tau).
  \end{align}Next consider a partition $\big(B_i\big)_{i\in\mathbb{N}}$ of $\mathbb{R}^d$ where $B_i$'s are Borel sets and fix $\xi_i\in B_i$. By Lemma \ref{lem:lip_holder} and Corollary \ref{cor:r-value}, $J$'s and $V_1$ are Lipschitz continuous in $x$ uniformly with respect to other variables. Therefore it is possible to choose the diameter of $B_i$'s small enough such that, for all $ p\in B_i, Y\in M(\tau)$ and $\beta\in \Delta_1(\tau)$,
   \begin{align}
  \label{eq:dpp-4}&|J(\tau,\xi_i; Y, \beta[Y])-J(\tau,p; Y, \beta[Y])|\le \epsilon
  \end{align}and 
  \begin{align}
  \label{eq:dpp-5}&|V_1(\tau,\xi_i)-V_1(\tau, p)|\le \epsilon.
  \end{align}
   We now define a strategy $\delta$ for the player $II$ as follows: for $(r,\omega)\in [t,T]\times \Omega_t$ and $Y\in M(t)$,
  \begin{align*}
       \delta[Y](r)=\begin{cases}
                   \beta^\epsilon[Y](r,\omega) &\text{if}~ r\in[t,\tau]\\
                   \sum_{i\in \mathbb{N}}
                   \chi_{B_i}(X_{t,x}(\tau))\beta^{\xi_i}[Y(\omega^{t,\tau})](r,\omega^{\tau, T})
                   &\text{if} ~r\in(\tau,T],
                   \end{cases}
     \end{align*}where $\pi(\omega)=(\omega^{t,\tau},\omega^{\tau, T})\in\Omega_{t,\tau}\times\Omega_\tau$, $Y(\omega^{t,\tau})(\cdot)\in M(\tau)$ is the $\mathcal{F}_{\tau,\cdot}$-predictable version mentioned in Lemma \ref{lem:control-breaup} and $X(\cdot)$ is the solution of \eqref{eq:SDG_dymcs} with the control pair $(Y, \beta^\epsilon[Y])$. If $Y$ is predictable, then by very definition $\delta[Y](\cdot)$ is also predictable. All other defining properties of an $r$-strategy is built within the definition of $\delta$. In other words, $\delta\in \Delta_1(t)$. 
   For $i \in \mathbb{N}$ such that $X(\tau)\in B_i$, we must have 
        \begin{align}
       \label{eq:dpp-added}V_1(\tau, \xi_i)\le J(\tau, X(\tau);\beta^{\xi_i}[Y(\omega^{t,\tau})], Y(\omega^{t, \tau})) +2\epsilon, 
     \end{align} for all $Y(\cdot)\in M(t)$ and for $P_{t,\tau}$-a.e. $\omega^{t,\tau}\in \Omega^{t,\tau}$. Also
      \begin{align*}
      &J(t,x; Y,\delta[Y])\\=&  E^{t,x}\Big\{\int_t^\tau f(s,X(s); Y(s),\delta[Y](s) )ds\\
      &\hspace{1cm}+\sum_{i\in \mathbb{N}}\chi_{B_i}(X(\tau))\big[\int_\tau^Tf(r,X(r);Y(r),\delta[Y](r))dr+g(X_T)\big]\Big\}\\
     =& E^{t,x}\Big\{\int_t^\tau f(s,X(s); Y(s),\beta^\epsilon[Y](s) )ds\\
    &\hspace{1cm}+\sum_{i\in \mathbb{N}}\chi_{B_i}(X(\tau))E\big[\int_\tau^Tf(r,X(r);Y(r),\delta[Y](r))dr+g(X_T)|\mathcal{F}_{t,\tau}\big]\Big\}\\
    =& E^{t,x}\Big\{\int_t^\tau f(s,X(s); Y(s),\beta^\epsilon[Y](s) )ds\\
      &\hspace{2cm}+\sum_{i\in \mathbb{N}}\chi_{B_i}(X(\tau))J\big(\tau, X(\tau); Y(\omega^{t,\tau}), \beta^{\xi_i}[Y(\omega^{t,\tau})]\big)\Big\}\\
\ge &E^{t,x}\Big\{\int_t^\tau f(s,X(s); Y(s),\beta^\epsilon[Y](s) )ds
      +\sum_{i\in \mathbb{N}}\chi_{B_i}(X(\tau))V_1(\tau,\xi_i)\Big\}-2\epsilon\\
\ge & E^{t,x}\Big\{\int_t^\tau f(s,X(s); Y(s),\beta^\epsilon[Y](s) )ds
      +V_1(\tau, X(\tau))\Big\}-3\epsilon,
      \end{align*} 
 where we have used Lemma \ref{markovprop} to deduce the third equality and the inequalities \eqref{eq:dpp-4}, \eqref{eq:dpp-5} and \eqref{eq:dpp-added} from above for the rest. Finally, invoke \eqref{eq:dpp-2} to conclude 
\begin{align*}
   W(t,x)\le   J(t,x; Y,\delta[Y])+4\epsilon
   \end{align*}
   implying 
\begin{align*}   
W(t,x)\le V_1(t,x)+4\epsilon, 
\end{align*} which leads to the desired conclusion \eqref{eq:dpp-sub} by letting $\epsilon\downarrow 0$.

\end{proof}

 As a corollary to Theorem \ref{thm:r_dpp}, we derive the following H\"{o}lder continuity estimate in time for $U_1$ and $V_1$.
\begin{cor}\label{cor:t-holder}
    There exists a constant $C > 0$, depending on the data, such that
\begin{align*}
    |V_1(t,x)-V_1(s,x)|+ |U_1(t,x)-U_1(s,x)|\le C|t-s|^{\frac 12}
\end{align*} for every $t,s \in [0, T]$ and  $x\in \mathbb{R}^d$.
\end{cor}
\begin{proof}
   Without loss of generality we may assume $0\le s < t\le T$ and $|t-s|< 1$. It follows as a consequence of \eqref{eq:dpp-sup} that
\begin{align*}
 &U_1(s,x)-U_1(t,x)\\ \le&\inf_{\alpha\in \Gamma_1(t)}\sup_{Z(\cdot)\in
    N(t)}E^{s,x}\Big\{\int_{s}^{t}f(r, X_{s,x}(r); \alpha[Z](r), Z(r))dr+U_1(t,X_{s,x}(t))-U_1(t,x)\Big\}\\
    \le & C\Big(|t-s|+E^{s,x}(|X_{s,x}(t)-x|)\Big),
\end{align*} where we have used the uniform Lipschitz continuity of $U_1(t, x)$ in $x$ and the boundedness of the data. Also,  the usual first-moment estimate (Lemma \ref{lem:A1}) for  SDE \eqref{eq:SDG_dymcs} implies 
\begin{align*}
 E^{s,x}(|X_{s,x}(t)-x|)\Big) \le K \sqrt{|t-s|}.
\end{align*} We combine the above estimates to conclude

\begin{align}
 \label{eq:holder-1} U_1(s,x)-U_1(t,x) \le C \sqrt{|t-s|}.
\end{align}
For any $Z(\cdot)\in N(t)$, we define $\tilde{Z}(\cdot)\in N(s)$ as 

\begin{align*}       
\tilde{Z}(r, \omega)=\begin{cases}
                    Z(t, \omega^{s,t}) & \text{if}~r\in[s,t]\\
                   Z(r, \omega^{s,t})
                   & \text{if}~r\in(t,T].
                   \end{cases}
     \end{align*} For any $\tilde{\alpha}\in\Gamma_1(s)$, we define $\alpha\in \Gamma_1(t)$ as 
\begin{align*}
 \alpha[Z] = \big(\tilde{\alpha}[\tilde{Z}]\big)(\omega^{s,t})\quad\quad(\text{as described in Lemma \ref{lem:control-breaup}}.)
\end{align*} and it is easy to see that $\alpha$ does not depend on $\omega^{s,t}$. For $\tilde{\alpha}\in \Gamma_1(s)$, we now have 

\begin{align*}
 J(s,x;  \tilde{\alpha}(\tilde{Z}), \tilde{Z} ) 
& = E^{s,x}\Big[\int_s^Tf(s,x;  \tilde{\alpha}[\tilde{Z}](r), \tilde{Z}(r) )dr + g(X_{s,x}(T))\Big]\\
& = E^{s,x}\Big[\int_s^tf(s,x;  \tilde{\alpha}[\tilde{Z}](r), \tilde{Z}(r) )dr +J(t,X_{s,x}(t);  \alpha[Z], Z )\Big]\\
& \ge -C\Big(|s-t|+E^{s,x}(|X_{s,x}(t)-x|)+J(t,x;  \alpha[Z], Z )\big),
\end{align*} where we have used Lemma \eqref{lem:lip_holder}. This implies

\begin{align*}
 &\sup_{Z\in N(s)} J(s,x;  \tilde{\alpha}[Z], Z ) \ge -C\big(|s-t|+E^{s,x}(|X_{s,x}(t)-x|)+ U_1(t,x)\\
\text{i.e.}\qquad & U_1(s,x)-U_1(t,x) \ge  -C \sqrt{|t-s|}, 
\end{align*}which, in combination with \eqref{eq:holder-1}, proves the H\"{o}lder continuity estimate for $U_1$. A similar set arguments works for $V_1$ as well.

\end{proof}

The Corollary \ref{cor:t-holder}, along with Lemma \ref{lem:lip_holder}, ensures that $U_1$ and $V_1$ are continuous in $(t,x)$. The inequalities \eqref{eq:dpp-sup} and \eqref{eq:dpp-sub} are referred to as dynamic {\em subprogramming} and {\em superprogramming} principles in the literature. These terminologies are also consistent with the following fact: \eqref{eq:dpp-sup} will imply that $U_1$ is a subsolution to the integro-PDE \eqref{eq:low-IPDE} and \eqref{eq:dpp-sub} will imply that $V_1$ is a supersolution to the integro-PDE \eqref{eq:upp-IPDE}. We prove this fact in the next theorem.

\begin{thm}\label{thm:r-sub-sup-solution}
     The $r$-upper value (resp. {\em lower value}) function $V_1$ (resp. $U_1$) is a viscosity supersolution (resp. {\em subsolution}) of \eqref{eq:low-IPDE}  (resp. \eqref{eq:upp-IPDE}).
   \end{thm}
  \begin{proof}

 We only prove that $U_1$ is a subsolution to the IPDE \eqref{eq:upp-IPDE}, the proof of $V_1$ being a supersolution is similar. Let $\varphi$ be a test function and $U_1-\varphi$ has a global maximum at $(t_0, x_0)\in [0,T)\times\mathbb{R}^d$. Since the value function $U_1$ is bounded, without any of generality, we may assume that the test function is bounded and has bounded derivatives up to second order. In view of definition \eqref{defi:visc}, we must show 
          \begin{align}
          \label{eq:sup-equality}\partial_t \varphi(t_0,x_0)+H^-\big(t_0, x_0, D\varphi(t_0, x_0), D^2\varphi(t_0, x_0),\varphi(t_0, \cdot)\big)\ge 0.
          \end{align}

 Let us assume the contrary, i.e. \eqref{eq:sup-equality} fails to hold, which means there exists a constant $\lambda > 0$ such that
          \begin{align}
          \label{eq:sup-equality-contra}-\partial_t \varphi(t_0,x_0)-H^-\big(t_0, x_0, D\varphi(t_0, x_0), D^2\varphi(t_0, x_0),\varphi(t_0, \cdot)\big)\ge \lambda > 0.
          \end{align} Set
             \begin{align*}
             \Lambda(t,x;y, z)&= \partial_t\varphi(t,x)+b(t,x; y,z).D\varphi(t,x)\\
              &\quad +\sum_{i,j} a_{ij}(t,x; y,z) \varphi_{x_i x_j}(t,x) + f(t,x; y,z).
             \end{align*}

We rewrite \eqref{eq:sup-equality-contra} as 
             \begin{align*}
              -\max_{z\in\mathcal{Z}}\min_{y\in\mathcal{Y}}\Big( \Lambda(t_0, x_0; y,z)+\mathcal{J}(t_0, x_0; y,z)\varphi\Big)\ge \lambda > 0,
             \end{align*}where the definition of  $\mathcal{J}$ is given immediately next to \eqref{eq:upp-IPDE}-\eqref{eq:low-IPDE}. Therefore, for each $z\in \mathcal{Z}$, there exists $y(z)\in\mathcal{Y}$ such that
             \begin{align}
                \label{eq:dyna_visc_contra}\Lambda(t_0, x_0; y(z),z)+\mathcal{J}(t_0, x_0; y(z),z)\varphi \le -\lambda.
             \end{align}

The continuity of $\Lambda+\mathcal{J}$ in $z$, uniformly in all other variables, implies
             \begin{align*}
             \Lambda(t_0, x_0; y(z),\zeta)+\mathcal{J}(t_0, x_0; y(z),\zeta)\varphi \le -\frac 23\lambda,
             \end{align*} for all $\zeta\in\mathcal{Z}\cap B(z, r)$ and some $r= r(z) > 0$. Obviously, $\big\{B(z, r(z)): z\in\mathcal{Z}\big\}$ is an open cover of $\mathcal{Z}$. The compactness of $\mathcal{Z}$ ensures the existence of a finite subcover. In other words, there are finitely many points $\{z_1, z_2,.....,z_n\}\subset\mathcal{Z}$,  $\{y_1, y_2,.....,y_n\}\subset\mathcal{Y}$ and $r_1,r_2,...r_n > 0$ such that $\mathcal{Z}\subset\bigcup_{i}B(z_i, r_i)$ and $$\Lambda(t_0, x_0; y_i,\zeta)+\mathcal{J}(t_0, x_0; y_i,\zeta)\varphi \le -\frac 23\lambda\quad\text{for}~\zeta\in B(z_i, r_i).$$
             Define $\varrho:\mathcal{Z}\rightarrow \mathcal{Y}$ as follows: $\varrho(z)= y_k$ if $z\in B(z_k, r_k)\backslash \cup_{i=1}^{k-1}B(z_i, r_i)$. Clearly, $\varrho$ is Borel measurable and
              \begin{align*}
              \Lambda(t_0, x_0; \varrho(z),z)+\mathcal{J}(t_0, x_0; \varrho(z),z)\varphi \le -\frac 23\lambda.
              \end{align*}

Now use the continuity of $\mathcal{J}+\Lambda$ in $(t,x)$, uniformly in $z$,  to find a $\vartheta > 0$ such that
                \begin{align}
                \label{eq:visco_contra}\Lambda(t, x; \varrho(z),z)+\mathcal{J}(t, x; \varrho(z),z)\varphi \le -\frac \lambda3
                \end{align}for all $z\in\mathcal{Z}$ and $\max(|t-t_0|, |x-x_0|)\le \vartheta$. The Borel measurable function $\varrho$ could be used to define a strategy $\hat{\alpha}$ for player $I$ on $[t_0, T]$ as 
                \begin{align*}
                \hat{\alpha}[Z](r, \omega)= \varrho(Z(r, \omega)),
                \end{align*}where $Z\in N(t_0)$ and $(r,\omega)\in[t_0,T]\times\Omega_{t_0}$.

 It is easy to see that $\hat{\alpha}$ maps a simple predictable process to a simple predictable process, therefore by standard limiting argument it follows that $\hat{\alpha}$ is an $r$-strategy. For the test function $\varphi$, we have
                 \begin{align*}
                    \varphi(\tau, X(\tau))-\varphi(t_0, x_0) \ge U_1(\tau, X(\tau))-U_1(t_0, x_0).
                 \end{align*} This, along with \eqref{eq:dpp-sup}, implies
      \begin{align}
  \label{eq:visco_dynamic}  \sup_{Z(\cdot)\in
    N(t_0)}E^{t_0,x_0}\Big\{\int_{t_0}^{\tau}f(s, X(s), \hat{\alpha}[Z](s), Z(s))ds+\varphi(\tau,X(\tau))-\varphi(t_0, x_0)\Big\}\ge 0,
     \end{align}where $X(\cdot)$ is the solution of \eqref{eq:SDG_dymcs} with the control pair $(\hat{\alpha}[Z](r), Z(r))$ and initial condition $(t_0, x_0)$. We use It\^{o}-L\'{e}vy formula and conclude
     \begin{align}
    \label{eq:visco_dynamic-1} &\sup_{Z(\cdot)\in
    N(t_0)}E^{t_0,x_0}\Big\{\int_{t_0}^{\tau}\big[\Lambda(r, X(r); \hat{\alpha}[Z](r), Z(r))\\\nonumber&\quad \quad\hspace{3cm}+\mathcal{J}(r, X(r); \hat{\alpha}[Z](r), Z(r))\varphi(r,X(r))\big]dr\Big\}\ge 0.
        \end{align} For $\epsilon > 0$, \eqref{eq:visco_dynamic-1} means that there exists $Z^\epsilon\in N(t_0)$ such that
        \begin{align}
       \label{eq:visco_dynamic-2} &E^{t_0,x_0}\Big\{\int_{t_0}^{\tau}\big[\Lambda(r, X(r); \hat{\alpha}[Z^\epsilon](r), Z^\epsilon(r))\\\nonumber&\quad \quad\hspace{1cm}\quad+\mathcal{J}(r, X(r); \hat{\alpha}[Z^\epsilon](r), Z^\epsilon(r))\varphi(r,X(r))\big]dr\Big\}\ge -\epsilon(\tau-t_0).
        \end{align} 
We indroduce \[G(r) = \Lambda(r, X(r); \hat{\alpha}[Z^\epsilon](r), Z^\epsilon(r))+\mathcal{J}(r, X(r); \hat{\alpha}[Z^\epsilon](r), Z^\epsilon(r))\varphi(r,X(r)), \] and rewrite \eqref{eq:visco_dynamic-2} as 
\begin{align*}
 E^{t_0,x_0}\Big\{\int_{t_0}^{\tau} G(r)\mathbf{1}_{||X(\cdot)-x_0||_{\infty}^\tau> \vartheta}dr+\int_{t_0}^{\tau} G(r)\mathbf{1}_{||X(\cdot)-x_0||_{\infty}^\tau\le \vartheta}dr\Big\} \ge -\epsilon(\tau-t_0).
\end{align*} \footnote{ {\bf Notation}: $||h||_\infty^\tau$ stands for the $L^ \infty$-norm of a function $h(\cdot)$ defined on $[t_0,\tau]$. }This implies, in combination with \eqref{eq:visco_contra}, that

\begin{align}
  \label{eq:visc_revis} E^{t_0,x_0}\big\{\int_{t_0}^{\tau} G(r)\mathbf{1}_{||X(\cdot)-x_0||_{\infty}^\tau> \vartheta}dr\big\}-\frac{\lambda}{3}(\tau-t_0)P_{t_0}\big(||X(\cdot)-x_0||_{\infty}^\tau\le \vartheta\big)\ge -\epsilon(\tau-t_0).
\end{align}
Next, apply Cauchy-Schwartz inequality relative to the measure space $\big([t_0,\tau]\times \Omega_{t_0}\big)$ and get
\begin{align*}
  &E^{t_0,x_0}\big(\int_{t_0}^{\tau} G(r)\mathbf{1}_{||X(\cdot)-x_0||_{\infty}^\tau> \vartheta}dr\big)\\
 \le & \big(E^{t_0,x_0}\int_{t_0}^{\tau}\mathbf{1}_{||X(\cdot)-x_0||_{\infty}^\tau\ge \vartheta}dr\big)^{\frac 12}. \big(E^{t_0,x_0}\int_{t_0}^{\tau}(G(r))^2dr\big)^{\frac{1}{2}}\\
 =& (\tau-t_0)^{\frac 12}\big(P_{t_0}(||X(\cdot)-x_0||_{\infty}^\tau\ge \vartheta\big)^{\frac 12}\big(E^{t_0,x_0}\int_{t_0}^{\tau}(G(r))^2dr\big)^{\frac{1}{2}}.
\end{align*} Since the test function $\varphi$ has bounded derivatives\footnote{ This is the only place where boundedness assumption. It is possible to work without this assumption, in which case we have to use the moment estiamtes for L\'{e}vy driven SDE.}, we must have
\begin{align*}
E^{t_0, x_0}\big( (G(r))^2\big) \le C(x_0,t_0, K, T),
\end{align*} where $C(x_0,t_0, K, T)$ is a constant depending only on $x_0, t_0, K, T$. Therefore \eqref{eq:visc_revis} gives
\begin{align}
 \label{eq:visc_revis_2}  E^{t_0,x_0}\big\{\int_{t_0}^{\tau} G(r)\mathbf{1}_{||X(\cdot)-x_0||_{\infty}^\tau\ge \vartheta}dr\big\} \le C(\tau-t_0)P_{t_0}\big(||X(\cdot)-x_0||_{\infty}^\tau\ge \vartheta\big)^{\frac{1}{2}}.
\end{align}

Next define the Ito-L\'{e}vy process $\xi(s)$  as
          \begin{align*}
            \xi(s)&= \int_{t_0}^{s}\sigma\big(r, X(r); \hat{\alpha}[Z^\epsilon](r),Z^\epsilon(r)\big)dW(r)\\&\quad+\int_{t_0}^s\int_{\mathbb{E}}\eta(r, X(r); \hat{\alpha}[Z^\epsilon](r),Z^\epsilon(r);w)\tilde{N}(dr,dw),
          \end{align*}which means $\xi(s)$ is a martingale. Note that

\begin{align*}
X(s)-x_0= \xi(s)+\int_{t_0}^sb(r, X(r); \hat{\alpha}[Z^\epsilon](r),Z^\epsilon(r))dr.
\end{align*} Hence 

\begin{align*}
  |X(s)-x_0|\le |\xi(s)|+\int_{t_0}^s|b(r, X(r); \hat{\alpha}[Z^\epsilon](r),Z^\epsilon(r))| dr,
\end{align*} which implies 
\begin{align*}
||X(\cdot)-x_0||_\infty^\tau & \le ||\xi(\cdot)||_\infty^\tau + K(\tau-t_0)\\
                                               & \le  ||\xi(\cdot)||_\infty^\tau + \frac{\vartheta}{2}
\end{align*} if 
$|\tau-t_0|< \frac{\vartheta}{2(K+1)}$. In conclusion, if $|\tau-t_0|< \frac{\vartheta}{2(K+1)}$ and $ ||\xi(\cdot)||_\infty^\tau < \frac{\vartheta}{2}$, then $||X(\cdot)-x_0||_\infty^\tau < \vartheta$. In other words, if $|\tau-t_0|< \frac{\vartheta}{2(K+1)}$, we must have 

\begin{align}
\label{eq:revis_visc_3.0} P_{t_0}\big(||X(\cdot)-x_0||_{\infty}^\tau\ge \vartheta\big) \le P_{t_0}\big(||\xi(\cdot)||_{\infty}^\tau \ge \frac{\vartheta}{2}\big).
\end{align}

By It\^{o}-isometry, along with  \ref{A1},\ref{A2} and \ref{A3},
\begin{align*}
 E^{t_0, x_0} \big(|\xi(\tau)|^2\big) = &\int_{t_0}^\tau E^{t_0,x_0}\big(|\sigma\big(r, X(r); \hat{\alpha}[Z^\epsilon](r),Z^\epsilon(r)\big)|^2\big)dr\\
                                    & + \int_{t_0}^\tau \int_{\mathbb{E}}E^{t_0,x_0}(|\eta(r, X(r); \hat{\alpha}[Z^\epsilon](r),Z^\epsilon(r);w)|^2)\nu(dw)dr\\
                                  \le& C(\tau-t_0).
\end{align*}

 Now we use Doob's martingale inequality and conclude
\begin{align}
 \label{eq:revis_visc_3} P_{t_0}\big(||\xi(\cdot)||_{\infty}^\tau \ge \frac{\vartheta}{2}\big)&\le \frac{C}{ \vartheta^2} E^{t_0, x_0} \big(|\xi(\tau)|^2\big)\\
                                 \nonumber                           & \le \frac{C}{ \vartheta^2} (\tau-t_0).
\end{align} 

Now combine the above inequalities \eqref{eq:visc_revis},\eqref{eq:visc_revis_2}, \eqref{eq:revis_visc_3.0},\eqref{eq:revis_visc_3} and divide  throughout by $\tau-t_0$ and get
                \begin{align*}
                  -N^\prime\sqrt{(\tau-t_0)}+\frac \lambda3 P_{t_0}\big(||X(\cdot)-x_0||_\infty^\tau \le \vartheta\Big)\le \epsilon
                \end{align*} for some constant $N^\prime$, which is clearly a contradiction to the stochastic continuity\footnote{See the proof of Lemma \ref{lem:stoc-continuity} to conclude that $X(\cdot)$ is stochastically continuous.} of $X(\cdot)$ for $(\tau-t_0)$ small.  

\end{proof}

The following lemma is a consequence of the above  theorem. 

\begin{lem}\label{lem:3.9}
Let \ref{A1},\ref{A2} and \ref{A3} hold. If $u$ and v are respectively the unique viscosity solutions of \eqref{eq:upp-IPDE}-\eqref{eq:term_cond}  and  \eqref{eq:low-IPDE}-\eqref{eq:term_cond}. Then, for every $(t,x)\in [0,T]\times \mathbb{R}^d$, we have 
\begin{align*}
u(t,x)\ge U_1(t,x)\ge U(t,x)\quad and\quad v(t,x)\le V_1(t,x)\le V(t,x).
\end{align*}
 
\end{lem}

\begin{proof}
The proof is immediate after combining Theorem \ref{thm:r-sub-sup-solution}, the comparison principle (Theorem \ref{thm:wellposed}) and Corollary \ref{cor:r-value} [$b.)$].

\end{proof}

 To complete the proof of Theorem \ref{thm:visc_connection}, it is still required to show $U(t,x)\ge u(t,x)$ and $V(t,x)\le v(t,x)$. Our proof of this requirement closely resembles \cite{Fleming:1989cg,Souganidis:1985sw}, except only the fact that the state processes can have discontinuities and the controls are predictable. The main idea is to approximate the value function with the help of ``piecewise constant'' strategies/controls, and next follows the  description of this methodology.

 Let $\pi = \{0= t_0< t_1<t_2<\cdot\cdot\cdot\quad <t_{n-1}< t_n= T\}$ be a partition of $[0,T]$ and $||\pi||$ be the norm of the this partition defined by $||\pi||=\max_{1\le i\le n}(t_i-t_{i-1})$. The concepts of the $\pi$-admissible strategies and $\pi$-admissible controls are defined as follows:
\begin{defi}[$\pi$-admissible controls]
  A $\pi$-admissible control $Y$ for player $I$ on $[t,T]$ is an admissible control with the following additional property:  If\linebreak$i_0\in\{0,1,.....,n-1\}$ is such that $t\in[t_0, t_{i_0+1}]$, then $Y(s)= y$ for $s\in[t, t_{i_0+1}]$ and $Y(s)= Y_{t_k}$ for $s\in(t_k, t_{k+1}]$ for $k= i_0+1,..., n-1$ where $Y_{t_k}$ is $\mathcal{F}_{t,t_k}$ measurable.
\end{defi}
The set of all $\pi$-admissible controls for player $I$ on $[t,T]$ is denoted by $M_{\pi}(t)$. A $\pi$-admissible control $Z(\cdot)$ for player $II$ is similarly defined and the set of all $\pi$-admissible controls for player $II$ is denoted by $N_\pi(t)$.

\begin{defi}[$\pi$-admissible strategies]
A $\pi$-admissible strategy $\alpha$ for player $I$ on $[t,T]$ is an $\alpha\in \Gamma(t)$ such that $\alpha[N(t)]\subset M_{\pi}(t)$ and the following conditions are satisfied: If $t\in [t_{i_0}, t_{i_0+1})$ then for every $Z(\cdot)\in N(t)$, $\alpha[Z]|_{[t, t_{i_0+1}]}$ is independent of $Z$. Furthermore, if $Z(\cdot)\approx\tilde{Z}(\cdot)$ on $[t, t_k]$ then $\alpha[Z](t_k+)= \alpha[\tilde{Z}](t_k+)$-$P_t$ a.s. for every $k\in \{i_0+1,....,n\}$. The set of all $\pi$-admissible strategies for player $I$ on $[t,T]$ is denoted by $\Gamma_{\pi}(t)$. The $\pi$-admissible strategies for player $II$ are similarly defined and the collection of all such strategies is denoted by $\Delta_{\pi}(t)$. 
\end{defi}

  For every $\psi\in W^{1,\infty}(\mathbb{R}^d)$, $(t,x)\in [0,T)$ and $\tau\in(t,T]$ define
 \begin{align}\label{eq:semigroup}
   S(t,\tau)\psi(x)= \inf_{y\in\mathcal{Y}}\sup_{Z\in N(t,\tau)}E^{t,x}\Big\{\int_{t}^\tau f(s, X(s); y, Z(s))ds+ \psi(X(\tau))\Big\},
 \end{align}where $N(t,\tau)$ is the set of all admissible control for player $II$ and $X(\cdot)$ is the solution of \eqref{eq:SDG_dymcs} with the control pair $(y, Z(\cdot))$  and initial condition $(t,x)$ on $[t,\tau]$. It is easy to see that if $\psi\in W^{1,\infty}$ then $S(t,\tau)\psi\in W^{1,\infty}$. Therefore, given the assumptions \ref{A1},\ref{A2} and \ref{A3},  the function $V_\pi: \mathbb{R}^d\times [0,T]\rightarrow \mathbb{R}$, given by $V_{\pi}(T, x)= g(x)$ and
 \begin{align}
 \label{eq:piecewise} V_{\pi}(t,x)= S(t, t_{i_0+1})\prod_{k=i_0+2}^{n}S(t_{k-1},t_k)g(x) \quad \text{if} ~t\in [t_{i_0}, t_{i_0+1}), 
 \end{align}is well defined. $V_\pi$  also has a stochastic game representation,  the precise form is stated as the following lemma.
   \begin{lem}\label{lem:game_represent}
    For every $(t,x)\in[0,T]\times \mathbb{R}^d$, the function $V_\pi$ has the following form:
    \begin{align}
       \label{eq:game-pi-characterization}V_\pi(t,x)=\sup_{\beta\in\Delta(t)}\inf_{Y\in M_\pi(t)} J(t,x; Y, \beta[Y]).
    \end{align}
   \end{lem}

\begin{proof}
 The main idea behind the proof is the same as for controlled diffusions (no jump)
    \cite[Proposition 2.3]{Fleming:1989cg}. The only exception remains into the fact that we need to take the discontinuities of the sample paths and predictability of the control processes into consideration. Following \cite[Proposition 2.3]{Fleming:1989cg}, the characterization \eqref{eq:game-pi-characterization} is  a consequence of the following fact: for every $(t,x)\in [0,T]\times \mathbb{R}^d$ and for every $\epsilon> 0$, there exists $\alpha_\epsilon\in \Gamma_\pi(t)$ and $\beta_\epsilon\in\Delta(t)$ such that for all $Y\in M_{\pi}(t)$ and $Z\in N(t)$,
\begin{align}
 \label{eq:pi-character-1}  J(t,x;\alpha_\epsilon[Z], Z)-\epsilon \le V_{\pi}(t,x) \le J(t,x;Y,\beta_\epsilon[Y])+\epsilon.
\end{align} The left half of the inequality in \eqref{eq:pi-character-1} implies the $'\le'$ in \eqref{eq:game-pi-characterization}.  To prove the other half of \eqref{eq:game-pi-characterization}, we can make use of the right half of \eqref{eq:pi-character-1} if we show that  for any $\beta\in \Delta(t)$, there exist  $Y^\epsilon\in M_{\pi}(t)$ and $Z^\epsilon\in N(t)$ such that 
\begin{align}
 \label{eq:pi-character-2} J(t,x;\alpha_\epsilon[Z^\epsilon], Z^{\epsilon}) = J(t,x;Y_{\epsilon},\beta[{Y^\epsilon}]). 
\end{align} The controls $Y^\epsilon$ and $Z^\epsilon$ could be defined as follows. Without loss of generality,  we may assume that $t= t_{i_0}$ and $z_0\in Z$. To this end,  we let $Y_{i_0}=\alpha[z_0]$ and $Z_{i_0}= \beta[Y_{i_0}]$ and successively define $Y_k\in M_{\pi}(t), \quad Z_{k}\in N(t)$ for $k= i_0+1,....., n$ as 
\begin{align*}
 Z_k =\beta[Y_k]\qquad \text{and}\qquad Y_k = \alpha_{\epsilon}[Z_{k-1}].
\end{align*} One must show $ Y_{k+1}\approx Y_{k}$ and $Z_{k+1}\approx Z_k$ on $[t_{i_0}, t_k]$ for $k = i_0+1,....., n-1$. We employ the method of induction. For $k=i_0+1$,  the conclusion is immediate from the definition of $\alpha_\epsilon$ and $\beta_\epsilon$ as $\alpha_\epsilon$ is independent of $z$-control on $[t_{i_0}, t_{i_0+1}]$. Next assume  $ Y_{k}\approx Y_{k-1}$ and $Z_{k}\approx Z_{k-1}$ on $[t_{i_0}, t_{k-1}]$. Note that $Y_{k+1} = \alpha_{\epsilon}[Z_{k}]$ and $Y_{k} = \alpha_{\epsilon}[Z_{k-1}]$, and from the definition of $\pi$-admissible strategies we have $Y(t_{k-1}+) = Y(t_{k}+)$. Since $Y_k, Y_{k+1}$ are constants on $(t_{k-1}, t_k]$, we must have $Y_{k+1}\approx Y_{k}$ on $[t_{i_0}, t_k]$. Consequently,  $Z_k =\beta[Y_k]$ and $Z_{k+1} =\beta[Y_{k+1}]$ implies $Z_{k+1}\approx Z_{k}$ on $[t_{i_0}, t_k]$. It is now straightforward to check \eqref{eq:pi-character-2} with $Z^\epsilon = Z_n$ and $Y^\epsilon = Y_n$.

For simplicity we choose $f= 0$. For $G\in W^{1,\infty}(\mathbb{R}^d)$, $y\in Y, ~~t\in[0,T]$ and $\tau\in(t,T]$, let 
\begin{align*}
 \phi(y;t,\tau;x, G) = \sup_{Z\in N(t,\tau)} E^{t,x} G(X(\tau)),
\end{align*}where $X(\cdot)$ is the solution of \eqref{eq:SDG_dymcs} with the control pair $(y, Z)$. In view of the assumptions \ref{A1},\ref{A2} and \ref{A3},  it follows as a consequence that $\phi(\cdot; t,\tau, \cdot,G)\in C_b(\mathbb{R}^d\times \mathcal{Y})\cap W^{1,\infty}(\mathbb{R}^d)$. Furthermore 
\begin{align*}
 S(t,\tau)G = \inf_{y\in\mathcal{Y}} \phi(y;t,\tau;x, G).
\end{align*}
If $t\in [t_{i_0, t_{i_0+1}}]$ for $i_0\in \big\{0,1,....., n-1\big\}$, let $G_n = g, ~G_j = S(t_j, t_{j+1})G_{j+1}$ for $j= i_0+1,....., n-1$ and $G_{i_0} = S(t, t_{i_0+1})G_{i_0+1}$, and thereby $G_{i_0}(x)= V_{\pi}(t,x)$

Next, partition the spaces $\mathbb{R}^d$ and $\mathcal{Y}$  respectively  into Borel sets $\{A_k: k=1,2,3,4,....\} $ and $\{B_\ell: \ell=1,2,..,L\}$ of diameters less than $\delta$, to be specified later.  Choose $x_k\in A_k$ and $y_{\ell}\in B_{\ell}$. Given $\gamma > 0$, we can choose $\delta$ small enough and $y_{kj}^* = y_{\ell(kj)}\in \mathcal{Y}$  for $k=1,....,n$ and $ j = i_0+1,...,n $, so that
\begin{align}
 \label{eq:pi-equival-3.0}\phi(y_{kj}^*;t_{j-1}, t_{j};x_k, G_j) < S(t_{j-1}, t_{j})G_{j}(x_k)+\gamma. 
 \end{align}

Furthermore, pick $Z_{k\ell j}\in N(t_{j-1}, t_j)$ such that,  for the controls $Y_s = y_\ell$ and $Z_{k\ell j}$,
\begin{align}
  \label{eq:pi-equival-3} E^{x_k,t_{j-1}}G_{j}(X^{k\ell j}_{t_j}) > \phi(y_\ell;t_{j-1}, t_j; x_k, G_{j})-\gamma.
\end{align} The superscripts signify the dependence of the solution $X^{k\ell j}$ of \eqref{eq:SDG_dymcs} on the intital data $(t_{j-1}, x_k)$ and the control $y_{\ell}$. We recall the description of canonical probability space to see 
\[E^{P_{t_{j-1}}} \equiv E^{t_{j-1}, x}. \] 

The strategies $\alpha_\epsilon$ and $\beta_\epsilon$ can now defined as 
\begin{align}
\label{eq:pi-strategy}\alpha_\epsilon[Z](r) = \chi_{[t_{i_0}, t_{i_0+1}]}(r)\sum_{k} y_{k i_0}^* \chi_{A_k}(x) + \sum_{j=i_0+1}^{n-1}\chi_{(t_j,t_{j+1}]}(r)\sum_k y_{kj}^* \chi_{A_k}(X(t_j)), 
\end{align} where the process $X(\cdot)$ is defined successively on intervals $[t,t_{i_0+1}],[t_j, t_j+1]$ for $j= i_0+1,...,n-1$ as the solution to \eqref{eq:SDG_dymcs} with $Y = \alpha_{\epsilon}[Z]$. For $Y\in M(t)$, we define

\begin{align}
 \label{eq:strategy-pi-2}\beta_{\epsilon}[Y] =&  \chi_{[t_{i_0}, t_{i_0+1}]}(r) \sum_{k,\ell} \tilde{Z}_{k\ell i_0}(r)\chi_{A_k}(x)\chi_{B_\ell}(Y(r))\\\nonumber&+\sum_{j=i_0+1}^{n-1}\sum_{k,\ell}\chi_{(t_j,t_{j+1}]}(r)\tilde{Z}_{k\ell j}(r)\chi_{A_k}(X(t_j))\chi_{B_\ell}(Y(r)),
\end{align} where $X(\cdot)$ is defined on successive intervals as the solution to \eqref{eq:SDG_dymcs} with $Z = \beta_\epsilon[Y]$ and $\tilde{Z}_{k\ell j}(r,\omega) =Z_{k\ell j}(r,\omega^{t_{j-1,T}})$.

For any $Z\in N(t)$ and $Y=\alpha_\epsilon[Z]$ or $Y\in M_\pi(t)$ and $Z=\beta_\epsilon[Y]$, it now follows that
\begin{align*}
 v_{\pi}(t,x) -J(t,x; \cdot, \cdot) &= G_{i_0}(x)-E^{P_t}(g(X(T)))\\
                                    &= \sum_{j= i_{0}+1}^{m}\big[E^{P_t}\big(G_{j-1}(X(t_{j-1}))\big)-E^{P_t}\big(G_{j}(X(t_{j}))\big)\big]\\
                                    & = E^{P_t}\Big[\sum_{j= i_{0}+1}^{m}\big[G_{j-1}(X(t_{j-1}))- E^{P_t}\big\{G_{j}(X(t_{j}))\big| \mathcal{F}_{t, t_{j-1}}\}\big]\Big],
\end{align*} where $J(t,x; \cdot,\cdot)$ stands for either $J(t,x;\alpha_\epsilon[Z], Z)$ or $J(t,x; Y, \beta_\epsilon[Y])$. To conclude \eqref{eq:pi-character-1} now we only need to show the following:

\begin{itemize}
 \item[] For any $Z \in N(t)$ and $Y= \alpha_\epsilon[Z]$
\begin{align}
 \label{eq:pi-equival-4}G_{j-1}(X(t_{j-1}))\ge E^{P_t}\big[ G(X(t_j))| \mathcal{F}_{t,t_{j-1}}\big]-\epsilon(t_{j}-t_{j-1})\qquad P^{t}-\text{a.s.},
\end{align}
 \item[] and  for any $Y \in M(t)$ and $Y= \beta_\epsilon[Y]$
\begin{align}
 \label{eq:pi-equival-5} E^{P_t}\big[ G(X(t_j))| \mathcal{F}_{t,t_{j-1}}\big]\ge G_{j-1}(X(t_{j-1})) -\epsilon(t_{j}-t_{j-1})\qquad P^{t}-\text{a.s}.
\end{align}
\end{itemize}
 Recall the  Lemma \ref{markovprop} and the discussion preceding it to see that the conditional expectations in \eqref{eq:pi-equival-4} and \eqref{eq:pi-equival-5} could now be considered as expectations with respect to $P_{t_{j-1}}$. Also, $X(\omega^{t,t_{j-1}}, \cdot)$ is a solution of \eqref{eq:SDG_dymcs} with initial condition $\big(t_{j-1}, X(t_{j-1})\big)$ for $P_{t, t_{j-1}}$-a.e $\omega^{t,t_{j-1}}$. Hence, for $X(t_{j-1})\in A_{k}$ and $Y(t_{j-1})\in B_{\ell}$, we have

\begin{align*}
 &\max\big\{ |X(t_{j-1})-x_k|, E^{P_{t_{j-1}}}|X(t_j)-X^{k\ell j}(t_j)|,  |G_{j-1}(X_{t_{j-1}}-G_{j-1}(x_k)|,\\ & \hspace{6cm} |E^{P_{t_{j-1}}}G(X(t_j))-E^{P_{t_{j-1}}}G(X(t_j^{k\ell j}))| \big\} = \mathcal{O}(\delta),
\end{align*} where $\mathcal{O}(\delta) \rightarrow 0$ as $ \delta\rightarrow 0$.  By \eqref{eq:pi-equival-3.0}, we must have 
\begin{align*}
 G_{j-1}(x_k) \ge E^{t,x} G_j(X^{k\ell(jk)j}(t_j))-\gamma
\end{align*} for each $k$. Again by Lemma \ref{lem:control-breaup}, $Z\in N(t)$ gives rise to $Z(\omega^{t,t_{j-1}})( \cdot)\in N(t_{j-1})$ and we conclude

\begin{align*}
 G_{j-1}(X(t_{j-1})) &\ge G(x_k)- \mathcal{O}(\delta)\\
                     &\ge G_j\big(X^{k\ell(jk)j}(t_j)\big)-\gamma-\mathcal{O}(\delta)\\
                     & \ge G_{j}(X(t_j))-\gamma - 2 \mathcal{O}(\delta),
\end{align*} which implies \eqref{eq:pi-equival-4} if  $\gamma+2 \mathcal{O}(\delta) \le \epsilon ||\pi||$. The reasoning for proving \eqref{eq:pi-equival-5} is similar. 
\end{proof}

\begin{lem}\label{lem:pi-continuity}
      Given the assumptions \ref{A1}, \ref{A2} and \ref{A3}, there exists a constant $C$, depending only on the data, such
      that
      \begin{align*}
         |V_\pi(t,x)|\le C\quad \text{and}\quad |V_\pi(t,x)-V_\pi(s, p)|\le C(|x-p|+|t-s|^{\frac 12})
      \end{align*} for all $(t,x), (s, p)\in [0, T]\times \mathbb{R}^d$.
    \end{lem}
    \begin{proof}
     The boundedness of $V_\pi$ is immediate. The Lipschitz continuity of $V_\pi$ in $x$ follows from the
      Lipschitz continuity of $J(\cdot\cdot\cdot\cdot)$'s, which is a consequence of Lipschitz continuity of the data.
      The argument is essentially the same as in Lemma \ref{lem:lip_holder}. The H\"{o}lder continuity estimate in $t$ follows by suitably mimicking the steps of Corollary \ref{cor:t-holder} after we invoke the stochastic game representation \eqref{eq:game-pi-characterization} of the $\pi$-value function $V_\pi$. We leave any further details of the proof to the interested readers.

\end{proof}

   Thanks to Arzela-Ascoli  theorem, in view of Lemma \ref{lem:pi-continuity}, the family $(V_\pi)_{\pi}$ will have locally uniformly convergent
     subsequences as $||\pi||\rightarrow 0$. In fact, it would not be difficult to show that any convergent subsequence of $V_\pi$, as
      $||\pi||\rightarrow 0$, will converge to the viscosity solution \eqref{eq:low-IPDE}-\eqref{eq:term_cond} and the uniqueness
       of the viscosity solution would imply that $\lim_{||\pi||\rightarrow 0}V_\pi$ exists.

    \begin{thm}\label{thm:pi-viscosity}
       Assume \ref{A1},\ref{A2}, \ref{A3} and $V_\pi$ is given by \eqref{eq:piecewise}. Then the limit
        $v=\lim_{||\pi||\rightarrow 0}V_\pi$ exists and it is the unique viscosity solution of
         \eqref{eq:low-IPDE}-\eqref{eq:term_cond}.
    \end{thm}

 \begin{proof}
In view of the preceding remark, we only need to show that any subsequential limit of $(V_\pi)_\pi$ as $||\pi||\rightarrow 0$ converges to the viscosity solution of \eqref{eq:low-IPDE}. Let $v$ be a locally uniform limit of a subsequence of the family $(V_\pi)_\pi$. We will show that $v$ is a subsolution of
 \eqref{eq:low-IPDE}, and the proof of $v$ being a supersolution is similar. The argument is classical in viscosity solution methods. If $\varphi$ is a test function and $v-\varphi$ has a strict global maximum at $(t_0, x_0)$, we wish to show that
\begin{align*}
 \varphi_t(t_0, x_0) + H^+(t_0, x_0, D\varphi, D^2\varphi, \varphi(t_0,\cdot)) \ge 0 .
\end{align*} Since the subsequence $V_\pi\rightarrow v$ locally uniformly as $||\pi||\rightarrow 0$, there exists $(t_\pi, x_\pi)$ such that $(t_\pi, x_\pi)\rightarrow (t_0, x_0)$ as $||\pi||\rightarrow 0$ and $(t_\pi, x_\pi)$ is a gobal maximum of $V_\pi-\varphi$. At the same time, if $t_\pi\in [t_{i_0}^\pi, t_{i_0+1}^\pi)$, then \eqref{eq:piecewise} gives $V_\pi(t_\pi, x_\pi) = S(t_\pi, t_{i_{0}+1}^\pi)V_\pi(t_{i_0+1}^\pi,\cdot)(x_\pi)$. Therefore

\begin{align*}
  \varphi(t_\pi, x_\pi) \le  S(t_\pi, t_{i_{0}+1}^\pi)\varphi( t_{i_0+1}^\pi,\cdot)(x_\pi).
\end{align*} The rest of the argument is now trivial, once we notice the following fact: For any test function $\varphi$,
 \begin{align*}
      \lim_{s\downarrow t}\frac{S(t,s)\varphi(\cdot)-\varphi(\cdot)}{s-t} = H^+(t,\cdot,D\varphi, D^2\varphi, \varphi(\cdot))
      \end{align*} holds as a consequence of It\^{o}-L\^{e}vy formula.

\end{proof}

We now piece together above results to conclude Theorem \ref{dpp} and Theorem \ref{thm:visc_connection}.

\begin{proof}[Proof of Theorem \ref{thm:visc_connection} ] By Lemma \ref{lem:game_represent}, we have 
\begin{align*}
V_\pi(t,x) \ge V(t,x)
\end{align*} for all $(t,x)\in [0,T]\times \mathbb{R}^d$ and any partition $\pi$ of $[0,T]$. Now we pass to the limit $||\pi||\rightarrow 0$ and invoke Theorem \ref{thm:pi-viscosity} to conclude $v\ge V$. The other half of the equality has already been obtained in Lemma \ref{lem:3.9}, therefore $V=v$. A similar line of arguments could be employed to conclude $U=u$.

\end{proof}
\begin{rem}
We can use Theorem \ref{thm:visc_connection} to conclude from Theorem \ref{thm:wellposed} that the value functions are H\"{o}lder continuous in time.
\end{rem}

\begin{proof}[Proof of Theorem \ref{dpp} ]
 It is sufficient to argue for \eqref{dpp-low}, proof of \eqref{dpp-up} is similar. Fix $\tau\in (0,T]$ and denote the right hand side of \eqref{dpp-low} by $V^*(t,x)$ for $(t,x)\in [0,\tau)\times \mathbb{R}^d$. Without any difficulty, in \ref{dpp-low},  it is enough to consider the controls $Y$'s and the strategies $\beta$'s
 defined on $[t,\tau]$, in place of the full interval $[t,T]$. Then by Theorem \ref{thm:visc_connection}, $V^*$ is the viscosity solution of \eqref{eq:low-IPDE} in $[0,\tau] $ with $V^*(\tau, x)= V(\tau,x)$. On the other hand,  $V(t,x)$ is the viscosity solution of the same problem. We conclude by uniqueness that $V=V^*$, and thereby proving the theorem. \end{proof}

\section{Proof of the verification theorem.}
      We begin this section with some essential technicalities  related to Lebesgue points of a measurable function with values
     in a Banach space. Some of these
      facts are discussed in \cite{Gozzi:2005bd}, rest are added here to tackle the additional  subtleties due to the nonlocal nature of Isaacs equation. 

   \begin{defi}\label{def:lebsg}
    Let $\mathbb{B}$ be a Banach space and $\gamma:[a,b]\rightarrow \mathbb{B}$ be measurable and Bochner
    integrable at the same time. A point $s\in [a, b]$ is said to be a right Lebesgue point of $\gamma$ if
    \begin{align*}
     \lim_{h\downarrow 0} \frac{1}{h} \int_{s}^{s+h} |\gamma(r)-\gamma(t)|_{\mathbb{B}}~dr = 0.
    \end{align*}
   \end{defi}The following result holds (see \cite{Diestel:1977hd} for the proof).
      \begin{lem}\label{lem:lbpts-measure}
      Let $\gamma:[a,b]\rightarrow \mathbb{B}$ be measurable and Bochner
    integrable. Then almost every point in $[a,b]$ is a right Lebesgue point of $\gamma$.
      \end{lem}
   From now onward we drop the subscript $t$ from $(\Omega_t, P_t)$  and simply write $(\Omega, P)$.  For a positive integer $n$, the next lemma ensures that any member of
  $L_{\mathcal{F}_{t,\cdot}}^1(t,T;\mathbb{R}^n)$ can be thought of as Bochner integrable when viewed as a map
  from $[t,T]$ to $L^1(\Omega; \mathbb{R}^n)$.
  \begin{lem}\label{lem:bochner}
   Any $\gamma\in L_{\mathcal{F}_{t,\cdot}}^1(t,T;\mathbb{R}^n)$ is Bochner integrable when regarded as a function from $[t,T]$ to
   $L^1(\Omega;\mathbb{R}^n)$.
  \end{lem}

\begin{proof}
    Except the null sets, the sigma algebra $\mathcal{F}_{t,T}$ is generated by a c\`{a}dl\`{a}g process. As an implication it follows that, excluding the null sets,  $\mathcal{F}_{t,T}$ is generated by a countable family of subsets of $\Omega$.  We combine this fact with separability of $\mathbb{R}^n$ and argue along the lines of (\cite{Doob:1994bn};
     p. 92) to conclude that $L^1(\Omega, \mathbb{R}^n)$ is separable. It follows from the definition of $ L_{\mathcal{F}_{t,\cdot}}^1(t,T;\mathbb{R}^n)$ that
      \[\int_t^T|\gamma(r,\cdot)|_{L^1(\Omega;
      \mathbb{R}^n)}dr< \infty.\]
      Hence, it only remains to show that the map $s\rightarrow \gamma(s,\cdot)$ is measurable as a function from $[t,T]$ to
       $L^1(\Omega; \mathbb{R}^n)$. To this end, we invoke the separability of $L^1(\Omega;\mathbb{R}^n)$ and conclude that measurability is equivalent to weak measurability. This means it is enough to prove $r\mapsto E(\gamma(r,\omega).\rho(\omega))$ is
         Lebesgue measurable for every $\rho\in L^\infty(\Omega; \mathbb{R}^n)$, but this is obvious once we apply
         the Fubini theorem as
         $$E\int_t^T|\gamma(r,\omega).\rho(\omega)|dr\le ||\rho||_{L^\infty(\Omega; \mathbb{R}^n)} \int_t^T|\gamma(r,\cdot)|_{L^1(\Omega;
      \mathbb{R}^n)}dr < \infty.$$
  \end{proof}

 A crucial technical difference we have here, in comparison with \cite{Gozzi:2005bd}, is the discontinuity
      of the state processes. However, these processes are stochastically continuous and the next lemma is consequence of this property.

    \begin{defi}[stochastic continuity]
     A stochastic process $(X(s))_{a\le s\le b}$, defined on a probability space $(\Omega, \mathcal{F}, P)$, is called
     stochastically continuous if for every $\epsilon > 0$ and $s\in[a,b]$
     $$\lim_{r\rightarrow s}P(|X(r)-X(s)|> \epsilon)= 0.$$
    \end{defi}
   We have the following lemma.

 \begin{lem}\label{lem:stoc-continuity}
  Let $X(\cdot)$ be the solution of \eqref{eq:SDG_dymcs} for a pair of admissible controls
  $(Y,Z)$, then $X(\cdot)$ is stochastically continuous and for a continuous function
   $\psi\in C_{p} \big([0,T]\times\mathbb{R}^d\big)$, for $p=1$  or  $p=2$, it holds that
   \begin{align*}
     \lim_{h\downarrow 0}E\Big(\frac
     1h\int_s^{s+h}\psi(r,X(r))dr\Big)= E\big(\psi(s,X(s))\big)
   \end{align*}
   \end{lem}
  \begin{proof}
The stochastic
continuity $X(\cdot)$ is a consequence of 
\begin{align*}
   X(\tau)-X(s) = \int_s^\tau b(r, X(r); Y(r), Z(r))dr+\int_s^\tau\sigma(r,X(r);Y(r), Z(r))dW(r)\\
   +\int_s^\tau\int_{\mathbb{E}}\eta(r, X(r^-);Y(r), Z(r);w)\tilde{N}(dr,dw),
\end{align*}and the assumptions on the data. It follows from the above realtionship that
\begin{align*}
 &E| X(\tau)-X(s)|^2\\ \le& 3E\big[ \int_t^\tau (b(r, X(r); Y(r), Z(r))) dr\big]^2+3\int_{s}^\tau {E}(\sigma^2(r, X(r); Y(r), Z(r))dr\\&+3\int_{s}^\tau \int_{\mathbb{E}}E(\eta^2(r, X(r); Y(r), Z(r);w))\nu(dw)dr \\
\le & K(|s-\tau|+|s-\tau|^2),
\end{align*} where we have used It\^{o}-L\'{e}vy isometry.  Therefore by Chebyshev's inequality, we have 

\begin{align*}
 P\big( | X(\tau)-X(s)| >\epsilon \big) \le \frac{K}{\epsilon^2}(|s-\tau|+|s-\tau|^2),
\end{align*} which proves the stochastic continuity of $X(r)$.

 Set $L(r) = \psi(r,X(r))$  and   $ F(r) =E[L(r)]$.  Assume that   $\psi\in C_{1} \big([0,T]\times\mathbb{R}^d\big)$. Thanks to the continuity of $\psi$,  $L(r)$ is also stochastically continuous. By It\^{o}-L\`{e}vy isometry, it is trivial to check that $X(r)$ has finite second second moment  for all $r\in [t,T]$,  and the bound is  independent of $r$.  Therefore $ F(r)$ is a bounded function on $[t,T]$.  
 
     \begin{align*}
     |F(r)-F(s)|&\le
                \int_{|L(r)-L(s)|\le\delta}|L(r)-L(s)|dP+\int_{|L(r)-L(s)|>\delta}|L(r)-L(s)|dP\\
                &\le \delta +E\big[ |L(r)-L(s)| {\bf  1}_{|L(r)-L(s)| >\delta}\big]\\
                & \le \delta + \big[E(|L(r)-L(s)|^2)\big]^{\frac 12} \big[P(|L(r)-L(s)| > \delta)\big]^{\frac{1}{2}}\\
                & \le  \delta + C \big[P(|L(r)-L(s)| > \delta)\big]^{\frac{1}{2}},
    \end{align*} which goes to $0$ as $s\rightarrow r$  and $\delta\rightarrow 0$. Therefore
    $F(r)$ is continuous in $r$ and hence every point is
    a lebesgue point.
      Now, after using Fubini's  theorem, dominated convergence theorem and the continuity of $F$ , we have
    \begin{align*}
     \lim_{h\downarrow 0}E\Big(\frac
     1h\int_s^{s+h}\psi(r,X(r))dr\Big) =\lim_{h\downarrow 0}\frac
     1h\int_{s}^{s+h} F(r)dr
       = F(s),
   \end{align*}  and this proves the lemma.
\end{proof}
\begin{proof}[Alternative proof of Lemma \ref{lem:stoc-continuity} for $p\in \{1,2\}$] 
  Since $X(\cdot)$ is C\`{a}dl\`{a}g ( i.e. right continuous ) and $\psi$ is continuous, one must have $\frac
     1h\int_s^{s+h}\psi(r,X(r))dr = \psi(s, X(s))$.  If $\psi \in C_p([0,T]\times \mathbb{R}^d)$, for $p =1$ or $2$, and since  $X(\cdot)$ has finite second moment, 
     we can apply dominated convergence theorem and conclude 
     
      \begin{align*}
     \lim_{h\downarrow 0}E\Big(\frac
     1h\int_s^{s+h}\psi(r,X(r))dr\Big) =\lim_{h\downarrow 0}\frac
     1h\int_{s}^{s+h} F(r)dr
       = F(s).
   \end{align*} 

\end{proof}
\begin{rem}
  The first proof of Lemma   \ref{lem:stoc-continuity}  also works for any $p\in \mathbb{N}$,  as it is possible to prove boundedness of any $2p$-th moment of $X(\cdot)$. It also says that $F(\cdot)$ is continuous. However, the alternative proof is much more compact and to the point. Though the second proof is enough conclude the lemma, but it only establishes right continuity of $F(\cdot)$. 
\end{rem}

  \begin{lem}\label{lem:lebsgue}
     Let the assumptions \ref{A1}, \ref{A2} and \ref{A3}  be satisfied and for
     $(t,x)\in[0,T]\times\mathbb{R}^d$,  $\big(Y(\cdot),Z(\cdot)\big)$
     be an admissible pair of controls, and $X(\cdot)$ be  the
     corresponding solution of \eqref{eq:SDG_dymcs}. Define the processes
     \begin{align*}
     z_1(r)&= b(r,X(r);Y(r),Z(r)),~~~z_2(r):=
     \sigma\sigma^T(r,X(r);Y(r),Z(r)),\\
     ~z_3(r, w)&=\eta(r,X(r);Y(r),Z(r);w)\eta^T(r,X(r);Y(r),Z(r);w),\\
     z_4(r,w) &=\eta(r,X(r);Y(r),Z(r); w).
     \end{align*}
    Then
    \begin{align}
    \label{lebesgue_point}& \lim_{h\downarrow 0}E\frac
    1h\int_s^{s+h}|z_i(r)-z_i(s)|dr~=0
    ~\text{for a.e. }~s\in[t,T]\quad\text{for}\quad i\in\{1,2\}, \\
     \label{lebesgue_point-2}&\lim_{h\downarrow 0}E\frac
    1h\int_s^{s+h}|z_4(r,w)-z_4(s,w)|dr~= 0
    ~\text{for a.e.}~s\in[t,T],~\text{for all}~w\in \mathbb{R}^m,
    \end{align} and
    \begin{align}
    \label{lebesgue_point-3}&\lim_{h\downarrow 0}E\frac
    1h\int_s^{s+h}\int_{\mathbb{E}}|z_3(r,w)-z_3(s,w)|\nu(dw)dr~= 0~\text{for a.e.}~s\in[t,T].
    \end{align}
   \end{lem}
\begin{proof}
 Fix $w\in \mathbb{R}^m$. Then, by \ref{A1},\ref{A2} and \ref{A3}, $z_1, z_2, z_4(\cdot, w)\in L^1_{\mathcal{F}_{t,\cdot}}(t, T; \mathbb{B})$ for $\mathbb{B}:= \mathbb{R}^d,\mathbb{S}^d, \mathbb{R}^d$ respectively.  By Lemma \ref{lem:bochner}, each of them is Bochner integrable, when viewed as $L^1(\Omega, \mathbb{B})$-valued maps and we conclude that the set of right Lebesgue points of $z_i$'s, as maps from $[t,T]$  to $L^1(\Omega, \mathbb{B})$,  is of  full measure in $[t,T]$. This gives \eqref{lebesgue_point}, and \eqref{lebesgue_point-2} except ``for all $w$'' part.  This implies that \eqref{lebesgue_point-2} holds for any `any countable dense subset of $\mathbb{R}^m$'. As 
$z_4(r,w)$ is continuous in $w$ uniformly in other entries, we have the full conclusion.

     We now prove \eqref{lebesgue_point-3}. Define  $\Omega^\prime= \Omega\times \mathbb{R}^m\backslash\{0\}$ and $\mu = P\otimes \nu$. Then $(\Omega^\prime,\mu)$ is measure space with the sigma algebra $\mathcal{F}\times \mathcal{B}(\mathbb{R}^m\backslash\{0\})$ and $z_3(r,\cdot)\in L^1\big((\Omega^\prime,\mu); \mathbb{S}^d\big)$. For the same reasoning detailed  in Lemma \ref{lem:bochner},  $z_3$ is Bochner integrable when considered as a map from $[t,T]$ to $L^1\big((\Omega^\prime,\mu); \mathbb{S}^d\big)$. Therefore we can apply Lemma \ref{lem:lbpts-measure} to conclude the result.
  
  \end{proof}

  \begin{lem}\label{lem:lpt}
     Let $\psi: [t,T]\times \mathbb{R}^d\mapsto \mathbb{R}^n$ be continuous and satisfy \eqref{eq:growth} with $k=1$.
      If \eqref{lebesgue_point-2} is satisfied for some $s\in[t,T]$, then
      \begin{align}
       \label{eq:lebesgue_new}\lim_{h\downarrow 0}  E\Big[\frac 1h\int_s^{s+h}\int_{\mathbb{E}}|\psi(r, X(r)+z_4(r,w))-\psi(s, X(s)+z_4(s,w))|\hat{\nu}(dw) dr\Big ] = 0,
      \end{align}
    where $\hat{\nu}(dw) := \min(1,|w|^2)\nu(dw)$ is a bounded Radon measure.
       \end{lem}
   \begin{proof}
       The proof is done in two steps.\\
       
       \noindent{\bf Step 1.} Let $\psi$ be Lipschitz continuous in $s$ and continuously differentiable in $x$ such that
       \begin{align*}
           |\psi(s,x)|, |D\psi(s,x)| \le C(1+|x|)\qquad \text{for}\quad (s,x)\in [0,T]\times \mathbb{R}^d.
       \end{align*}
        Then
           \begin{align*}&|\psi(r, X(r)+z_4(r,w))-\psi(s, X(s)+z_4(s,w))|\\
           \le &C\big[ |r-s|+\big(|X(r)|+|X(s)|+|z_4(r,w)|+|z_4(s,w)| \big)\\&\hspace{2cm}\times\big(|X(r)-X(s)|+|z_4(r,w)-z_4(s,w)|\big)\big].
           \end{align*} By It\^{o}-L\`{e}vy isometry, $X(r)$ has bounded second moment for all $r\in[t,T]$.  With this, we now use boundedness of $\eta$ and Cauchy-Schwartz inequality to obtain 
           \begin{align*}
             &\ E\frac 1h\int_s^{s+h}\int_{\mathbb{E}}|\psi(r, X(r)+z_4(r,w))-\psi(s, X(s)+z_4(s,w))|\hat{\nu}(dw) dr\\
             \le & C\frac 1h\int_s^{s+h}\int_{\mathbb{E}}\big[|r-s|+(E|X(r)-X(s)|^2)^{\frac 12} \big) \hat{\nu}(dw) dr \\
             &+ C^\prime\frac 1h\Big[ E\int_s^{s+h}\int_{\mathbb{E}} |z_4(r,w)-z_4(s,w)| \hat{\nu}(dw) dr \Big]^{\frac12}.
           \end{align*} We now use \eqref{lebesgue_point-2} and dominated convergence theorem to pass to the limit $h\downarrow 0$ and conclude \eqref{eq:lebesgue_new}.  \\
           
           \vspace{.05cm}
              \noindent{\bf Step 2.}  If  $\psi: [t,T]\times \mathbb{R}^d\mapsto \mathbb{R}^n$ be continuous and satisfies \eqref{eq:growth} with $k=1$. Then, for every $\epsilon> 0$, there exists $\psi_\epsilon$ satisfying the conditions of  the first step such that $|\psi(s,x)-\psi_\epsilon(s,x)|\le C \epsilon(1+|x|).$ Then
             
     \begin{align*}
        &\quad\lim_{h\downarrow 0}E\frac 1h\int_s^{s+h}\int_{\mathbb{E}}|\psi(r, X(r)+z_4(r,w))-\psi(s, X(s)+z_4(s,w))|\hat{\nu}(dw) dr\\
        &\le \lim_{h\downarrow 0} E\frac 1h\int_s^{s+h}\int_{\mathbb{E}}|\psi_\epsilon(r,X(r)+z_4(r,w))-\psi_\epsilon(s, X(s)+z_4(s,w))|\hat{\nu}(dw)
         dr\\&\hspace{10cm}+2C\hat{\nu}(\mathbb{E})\epsilon,
       \end{align*} where the right hand side goes to zero as $\epsilon \rightarrow 0$.       

          \end{proof}

 \begin{proof}[Proof of Theorem \ref{thm:verification}]
   For the sake of simplicity we use the following abbreviated notations:
      \begin{align*}
     f^*(s)&= f(s,X^*(s); Y^*(s), Z^*(s)), \quad g^*(s)= g(X^*(s)),\\
     b^*(s)&= b(s,X^*(s);Y^*(s), Z^*(s)),\quad
     \sigma^*(s)=\sigma^*(s,X^*(s);Y^*(s), Z^*(s)),\\
     &\quad\quad \eta^*(s,w)=\eta^*(s,X^*(s);Y^*(s), Z^*(s);w).
   \end{align*}

Choose $\tau \in[t, T]$ so that \eqref{lebesgue_point} holds at $\tau$ for $z_1(\cdot)=b^*(\cdot)$,
   $z_2(\cdot)=\sigma^*(\cdot){\sigma^*(\cdot)}^T$; \eqref{lebesgue_point-3} holds for $z_3(\cdot,w)= \eta^*(\cdot, w)\eta^*(\cdot, w)^T$ and
   \eqref{lebesgue_point-2} holds for $z_4(\cdot, w)=\eta^*(\cdot,w)$. By Lemma \ref{lem:lebsgue}, the set of all such $\tau$'s in $[t,T]$
   is of full measure. From Section 2, recall the identification
     
 \begin{align*}
      \Omega = \Omega_{t,\tau}\times \Omega_\tau\quad\text{and}\quad P=P_1\otimes P_2,
     \end{align*} where $P_1= P_{t,\tau}$ and $P_2=P_{\tau}$. We use $E^{P_1}$ and $E^{P_2}$  for expectations with respect to $P_1$ and $P_2$ respectively. For any $\mathcal{F}_{t,\tau}$ measurable function
     $\varphi(\omega)$ on $(\Omega, P)$ we have

\begin{align*}
      \chi(\omega)=E^{P}[\chi(\omega_1,\omega_2)|\mathcal{F}_{t,T}]=
      E^{P_1\times P_2}[\varphi(\omega_1,\omega_2)|\mathcal{F}_{t,\tau}]=E^{P_2}(\varphi(\omega_1,\omega_2)),~P_1-a.s.
     \end{align*}i.e.  $P_1$-a.s.  $\varphi(\omega)$ is deterministic in $(\Omega_\tau, P_2)$ , where $\omega= (\omega_1, \omega_2)\in \Omega_{t,\tau}\times\Omega_{\tau}$. Therefore $X^*(\tau),
     \Phi^1(\tau), p^1(\tau), q^1(\tau)$ and $ Q^1(\tau)$ are all $P_1-a.s.$ deterministic in $(\Omega_\tau, P_2)$. It has
      already been pointed out in Section 3 that,  on $(\Omega_\tau, P_2)$,   $X^*(s)_{s\ge \tau}$ has the 
       dynamics

 \begin{align}
      \label{eq:new-dynamics}X^*(s)=& X^*(\tau)+\int_\tau^s b^*(r)dr+\int_\tau^s\sigma^*(r)dW^\tau(r)\\&\nonumber+\int_\tau^s\int_{\mathbb{E}}\eta(X(r^-);Y(r), Z(r);w)
      \tilde{N}^\tau(dr,dw)
      \end{align} $P_1$-a.s,
 where $(W^\tau,\tilde{N}^\tau)$ be the process driving the dynamics on $(\Omega_\tau, P_2)$.  We wish to apply Ito-L\'{e}vy formula  to $\Phi^1_\tau(s, X^*(s))$ on $(\Omega_\tau, P_2)$,
        relative to the dynamics \eqref{eq:new-dynamics}.  Firstly, $\Phi_\tau^1$ is deterministic on $\Omega_\tau$ $(P_1-a.s.)$ and secondly it is $C^{1,2}$, hence It\^{o}-L\`{e}vy formula is applicable. 
        To this end, we simply write $\phi$ for $\Phi^1_\tau$  and apply Ito-L\`{e}vy formula to conclude, for $h> 0$,

\begin{align*}
         &\phi(\tau+h, X^*(\tau+h))-\phi(\tau,X^*(\tau))\\
         =&\int_{\tau}^{\tau+h}\Big[\partial_s\phi(r,X^*(r))+\langle D\phi(r,X^*(r)); b^*(r)\rangle
         +\frac 12
         \text{Tr}\big(\sigma^*{\sigma^*}^T(r)D^2\phi(r,X^*(r))\big)\Big]
         dr\\
         &+\int_{\tau}^{\tau+h}\int_{\mathbb{E}}\Big(\phi(r,X^*(r)+\eta^*(r,w))
         -\phi(r,X^*(r))-\eta^*D\phi(r, X^*(r))\Big)\nu(dw)dr\\
          &+\int_{\tau}^{\tau+h}\langle
         D\phi(r,X^*(r);\sigma^*(r))dW^\tau(r)\rangle
         \\&+\int_{\tau}^{\tau+h}\int_{\mathbb{E}}\big[\phi(r,X^*(r^-)+\eta(\cdot\cdot\cdot\cdot))-\phi(r,X^*(r^-))\big]\tilde{N}^\tau(dr,dw),
         \quad P_1-a.s.
         \end{align*}

 We divide the above equality by $h$, apply $E^{P_2}$ and use the fact that $u-\phi$ has a global
         maximum at $(\tau, X^*(\tau))$ to obtain
       \begin{align}
         \label{eq:ito-p2-average}&E^{P_2}\frac 1h\big[u(\tau+h, X^*(\tau+h))-u(\tau,X^*(\tau))\big]\\
        \nonumber \le &E^{P_2}\frac 1h\Big[\int_{\tau}^{\tau+h}\big[\partial_s\phi(r,X^*(r))+\langle D\phi(r,X^*(r)); b^*(r)\rangle
         \\\nonumber&\hspace{2.5cm}+\frac 12
         \text{Tr}\big(\sigma^*(r)D^2\phi(r,X^*(r))\sigma^*(r)\big)\big]
         dr\\
         \nonumber&\hspace{1cm}+\int_{\tau}^{\tau+h}\int_{\mathbb{E}}\Big(\phi(r,X^*(r)+\eta^*(r,w))-\phi(r,X^*(r))\\
         \nonumber&\hspace{3.8cm}-\eta^*(r,w)D\phi(r,X(r))\Big)
         \nu(dw)dr\Big],
         \end{align}$P_1$-- a.s. We are interested in passing to the limit $h\downarrow 0$ in \eqref{eq:ito-p2-average},
          and we do so by separately considering each term from the right-hand side. We begin with the simplest one.

  Since $\partial_s \phi$ is continuous, we can apply Lemma \ref{lem:stoc-continuity}~ for
            $\psi= \partial_s \phi$ [relative to $(\Omega_\tau, P_2)$] and get
            \begin{align}
              \label{eq:est-1}\lim_{h\downarrow 0} E^{P_2}\frac 1h \int_\tau^{\tau+h} \partial_s \phi(r, X^*(r)) dr
              =\partial_s\phi(\tau, X^*(\tau)),\quad P_1-a.s.
            \end{align}
  Next we treat the term resulting directly from the presence of jumps. Let us denote
   \begin{align*}
    I(r) =& E^{P_2}\Big[\int_{\mathbb{E}}\Big(\phi(r,X^*(r)+\eta^*(r,w))-\phi(r,X^*(r))
         \\&\hspace{3.5cm}-\eta^*(r,w)D\phi(r,X(r))\Big)
         \nu(dw)\Big].
   \end{align*}We use Fubini's theorem to justify the identity 
   \begin{align}\label{eq:fub-identity}
   &\int_{\mathbb{E}}\Big(\phi(r,X^*(r)+\eta^*(r,w))-\phi(r,X^*(r))
         -\eta^*(r,w)D\phi(r,X(r))\Big)
         \nu(dw)\\
  \nonumber = &\int_0^1(1-\rho)\big(\int_{\mathbb{E}}\langle D^2\phi(r, X^*(r)
  +\rho\eta^*(r,w))\eta^*(r,w);\eta^*(r,w)\rangle\nu(dw)\big)d\rho\\
   \nonumber = &\int_0^1(1-\rho)\big(\int_{\mathbb{E}} \text{Tr}\big[\eta^*(r,w)\eta^*(r,w)^TD^2\phi(r, X^*(r)
  +\rho\eta^*(r,w))\big]\nu(dw)\big)d\rho,
   \end{align}where we have also used properties of trace. Therefore

 \begin{align}
   \label{eq:est-2}&I(r)-I(\tau)\\
   \nonumber= &\int_0^1(1-\rho)\Big(\int_{\mathbb{E}} \text{Tr}\big(\big[\eta^*(r,w)\eta^*(r,w)^T-\eta^*(\tau,w)\eta^*(\tau,w)^T\big]\big)\\
   \nonumber&\hspace{5cm}\times D^2\phi(\tau, X^*(\tau)
  +\rho\eta^*(\tau,w))\big)\nu(dw)\Big)d\rho\\
  \nonumber & +\int_0^1(1-\rho)\Big(\int_{\mathbb{E}} \text{Tr}\big(\eta^*(r,w)\eta^*(r,w)^T\big[D^2\phi(r, X^*(r)
  +\rho\eta^*(r,w))\\
 \nonumber &\hspace{5cm}-D^2\phi(\tau, X^*(\tau)
  +\rho\eta^*(\tau,w))\big]\big)\nu(dw)\Big)d\rho.
   \end{align}

Since $D^2\phi$ is continuous and $\rho\eta^*(r,w)$ satisfies \eqref{lebesgue_point} at $\tau$,
    upon denoting $$I^\prime(r,\tau)=\int_{\mathbb{E}} |D^2\phi(r, X^*(r)
  +\rho\eta^*(r,w))-D^2\phi(\tau,X^*(\tau)
  +\rho\eta^*(\tau,w))|\hat{\nu}(dw)$$ where $\hat{\nu}(dw)= \min(1,|w|^2)\nu(dw)$,  by Lemma \ref{lem:lpt}, we must have
   \begin{align}
     \label{eq:est-3}0=& \lim_{h\downarrow 0}E\frac 1h\int_\tau^{\tau+ h} I^\prime(r,\tau) dr\\
  \nonumber    =&\lim_{h\downarrow 0}E\Big[E\big[\frac 1h\int_\tau^{\tau+ h}I^\prime(r,\tau) dr\big|\mathcal{F}_{t,\tau}\big]\Big]\\
       \nonumber=& \lim_{h\downarrow 0}E^{P_1}\Big[E^{P_2}\big[\frac 1h\int_\tau^{\tau+ h}I^\prime(r,\tau) dr\big]\Big],
   \end{align} where we have also used Lemma \ref{markovprop}. Also note that $z_3(\cdot,w)= \eta^*{\eta^*}^T(\cdot,w)$ satisfies \eqref{lebesgue_point-3} at $\tau$ and hence, 
    after a similar reasoning as above, we have
   \begin{align}
     \label{eq:est-4}0=\lim_{h\downarrow}E^{P_1}\Big[E^{P_2}\frac 1h\int_\tau^{\tau+ h}
     \int_{\mathbb{E}}\Big|\text{Tr}\big[{\eta^*}{\eta^*}^T(r,w)-\eta^*{\eta^*}^T(\tau,w)\big]\Big |\nu(dw)dr\Big].
   \end{align} 

  We now recall the growth properties of $\eta$, $\phi$ and use them along with Fubini's theorem to get
   \begin{align}
   \nonumber&\hspace{.5cm}E^{P_2}\frac 1h \int_\tau^{\tau+h} |I(r)-I(\tau)|dr\\
     \label{eq:est-5}&\le N(\tau)\Big[\frac 1h E^{P_2}\int_\tau^{\tau+h}\int_{\mathbb{E}}|\text{Tr}
    \big[\eta^*{\eta^*}^T(r,w)-\eta^*{\eta^*}^T(\tau,w)\big]|\nu(dw)dr\Big]\\
   \nonumber  & \quad+N\int_0^1(1-\rho)E^{P_2}\big[\frac 1h\int_\tau^{\tau+ h}I^\prime(r,\tau)dr\big]d\rho,
     \end{align}where $N$ depends on data and $N(\tau)$ depends also on $X^*(\tau)$ .
      Fix any sequence $\{h_j\}_j$ such that $h_j\downarrow 0$. From \eqref{eq:est-3}\ and \eqref{eq:est-4}, we observe that
      $E^{P_2}\big[\frac 1{h_j}\int_\tau^{\tau+ h_j}I^\prime(r,\tau) dr\big]$ and $E^{P_2}\frac 1{h}_j\int_\tau^{\tau+h_j}\int_{\mathbb{E}}\big(|\text{Tr}
    \big[\eta^*{\eta^*}^T(r,w)-\eta^*{\eta^*}^T(\tau,w)\big]|\nu(dw)dr$ goes to $0$ in $L^1(\Omega_{t,\tau}, P_1)$.
    This means there exists a subsequence $(h_l)_l$ of $(h_j)_j$ such that
 \begin{align*}
     &\lim_{h_l\downarrow 0}E^{P_2}\big[\frac 1{h_l}\int_\tau^{\tau+ h_l}I^\prime(r,\tau)dr\big]= 0,\\
 &\lim_{h_l\downarrow 0} E^{P_2}\frac 1{h}_l\int_\tau^{\tau+h_l}\int_{\mathbb{E}}\big(|\text{Tr}
    \big[\eta^*{\eta^*}^T(r,w)-\eta^*{\eta^*}^T(\tau,w)\big]|\nu(dw)dr=0
    \end{align*} $P_1$-a.s. We use the above equalities along
     with \eqref{eq:est-5} to conclude
     \begin{align}
     \label{eq:main-est-2}\lim_{h_l\downarrow 0 }E^{P_2}\frac 1{h_l} \int_\tau^{\tau+h_l} |I(r)-I(\tau)|dr = 0, \quad P_1-a.s.
    \end{align}
    Next we treat the term resulting from the drift of the dynamics.

  \begin{align}
      \label{second_term}&E^{P_2}\frac 1h\int_{\tau}^{\tau+h}\big|\langle D\phi(r,X^*(r));
      b^*(r)\rangle -\langle D\phi(\tau,X^*(\tau));
      b^*(\tau)\rangle\big | dr\\
    \nonumber =& ||b||_\infty E^{P_2}\frac 1h\int_{\tau}^{\tau+h}| D\phi(r,X^*(r))-D\phi(\tau,X^*(\tau))| dr
     \\\nonumber&\hspace{1cm}+|\langle D\phi(\tau,X^*(\tau))|E^{P_2}\frac 1h\int_{\tau}^{\tau+h}|b^*(r)-
      b^*(\tau)| dr.
      \end{align} We use continuity of $D\phi$ to apply Lemma \ref{lem:stoc-continuity} and get
         \begin{align}
         \label{eq:est-6} \lim_{h\downarrow 0} E^{P_2}\frac 1h\int_{\tau}^{\tau+h}| D\phi(r,X^*(r))-D\phi(\tau,X^*(\tau))| dr=0, \quad P_1-a.s.
         \end{align}
       Also, by the choice of $\tau$, $z_1(\cdot)=b^*(\cdot)$ satisfies \eqref{lebesgue_point} i.e.
       \begin{align*}
        0&=\lim_{h\downarrow 0}E\frac
      1h\int_{\tau}^{\tau+h}|b^*(r)-b^*(\tau)|dr\\
       &=\lim_{h\downarrow 0}E\Big(E\big[\frac
      1h\int_{\tau}^{\tau+h}|b^*(r)-b^*(\tau)|dr\big|\mathcal{F}_{t,\tau}\big]\Big)\\
        &= \lim_{h\downarrow 0}E^{P_1}\Big(E^{P_2}\big[\frac
      1h\int_{\tau}^{\tau+h}|b^*(r)-b^*(\tau)|dr\big]\Big),
       \end{align*} which implies that $E^{P_2}\frac
      1h\int_{\tau}^{\tau+h}|b^*(r)-b^*(\tau)|dr$ converges to $0$ in $L^1((\Omega_{t,\tau}, P_1);\mathbb{R})$.

  Therefore, there exists subsequence $(h_{l^\prime})$ of $(h_l)$ such that
       \begin{align}
           \label{eq:est-7}\lim_{h_{l^\prime}\downarrow 0} E^{P_2}\frac
      1{h_{l^\prime}}\int_{\tau}^{\tau+h_{l^\prime}}|b^*(r)-b^*(\tau)|dr= 0.
       \end{align} We combine \eqref{eq:est-6}-\eqref{eq:est-7} to conclude, $P_1-a.s.$
       \begin{align}
       \label{eq:main-est-3}\lim_{h_{l^\prime}\downarrow 0}E^{P_2}\frac 1{h_{l^\prime}}\int_{\tau}^{\tau+h}\langle D\phi(r,X^*(r));
      b^*(r)\rangle-\langle D\phi(\tau,X^*(\tau));
      b^*(\tau)\rangle dr= 0.
       \end{align} Finally, we consider the term resulting from the diffusions. We routinely use formulas involving
        traces of matrices to write

 \begin{align*}
        &E^{P_2}\Big[ \frac 1h \int_\tau^{\tau +h}\frac 12\text{Tr}\Big(\sigma^*(r){\sigma^*(r)}^TD^2\phi(r,X^*(r))-\sigma^*(\tau)
        {\sigma^*(\tau)}^TD^2\phi(\tau,X^*(\tau))\Big)dr\Big]\\
        =&E^{P_2}\Big[ \frac 1h \int_\tau^{\tau +h}\frac 12\text{Tr}\big(\sigma^*(r){\sigma^*(r)}^T[D^2\phi(r,X^*(r))
        -D^2\phi(\tau,X^*(\tau))]\big)dr\Big]\\
        &+E^{P_2}\Big[ \frac 1h \int_\tau^{\tau +h}\frac 12\text{Tr}\big([\sigma^*(r){\sigma^*(r)}^T
        -\sigma^*(\tau){\sigma^*(\tau)}^T]D^2\phi(r,X^*(r))\big)dr
        \end{align*} and then we invoke the same set of arguments used above to conclude that there exists a subsequence
        $(h_{l^{\prime\prime}})$ of $(h_{l^\prime})$ such that, $P_1-a.s.$,
        \begin{align}
        \label{eq:main-est-4} \lim_{h_{l^{\prime\prime}}\downarrow 0}E^{P_2}\Big[ \frac 1{h_{l^{\prime\prime}}} \int_\tau^{\tau +h_{l^{\prime\prime}}}
        \frac 12\text{Tr}\Big(\sigma^*{\sigma^*}^T(r)D^2\phi(r,X^*(r))-\sigma^*
        {\sigma^*}^T(\tau)D^2\phi(\tau,\cdot)\Big)dr\Big]= 0.
        \end{align}

 Summing up \eqref{eq:est-1},\eqref{eq:main-est-2},\eqref{eq:main-est-3}  and \eqref{eq:main-est-4},
        with \eqref{eq:diff-sup-integral} in mind, for any sequence $(h_j)\downarrow 0$ there exists a subsequence $(h_{l^{\prime\prime}})$ such that, $P_1-a.s.$,
           \begin{align}
            \label{eq:est-8}&\lim_{h_{l^{\prime\prime}}\downarrow 0}E^{P_2}\Big[\frac 1{h_{l^{\prime\prime}}}\int_\tau^{\tau+h}\big[
            \partial_s\phi(r,X^*(r))+\langle D\phi(r,X^*(r)); b^*(r)\rangle
         \\\nonumber&\hspace{3.5cm}+\frac 12
         \text{Tr}\big(\sigma^*(r)D^2\phi(r,X^*(r))\sigma^*(r)\big)\big]
         dr\\
        \nonumber&\hspace{1.5cm}+\int_{\tau}^{\tau+h}\int_{\mathbb{E}}\big(\phi(r,X^*(r)+\eta^*(r,w))-\phi(r,X^*(r))\\
         \nonumber&\hspace{3.8cm}-\eta^*(r,w)D\phi(r,X^*(r))\big)
         \nu(dw)dr
            \Big]\\
          \nonumber  =& \partial_s\phi(\tau,X^*(\tau))+\langle D\phi(\tau,X^*(\tau)); b^*(\tau)\rangle+
         \frac 12
         \text{Tr}\big(\sigma^*(\tau)D^2\phi(\tau,X^*(\tau))\sigma^*(\tau)\big)
         \\
        \nonumber &+\int_{\mathbb{E}}\big(\phi(\tau,X^*(\tau)+\eta^*(\tau,w))-\phi(\tau,X^*(\tau))
        -\eta^*(\tau,w)D\phi(\tau,X^*(\tau))\big)
         \nu(dw).
           \end{align}

  Passing to the limit along a suitable subsequence for which $\limsup$ is achieved in \eqref{eq:ito-p2-average}
            and
            using \eqref{eq:diff-sup}, we have
           \begin{align}
         \nonumber &\limsup_{h\downarrow 0} E^{P_2}\frac 1h\big[u(\tau+h, X^*(\tau+h))-u(\tau,X^*(\tau))\big]\\
          \label{eq:est-final-1}\le& p^1(\tau)+\langle q^1(\tau), b^*(\tau)\rangle+\frac12\text{Tr}(\sigma^*{\sigma^*}^T(\tau)Q^1(\tau))
           \\\nonumber& + \mathcal{J}\big(\tau,X^*(\tau); Y^*(\tau), Z^*(\tau)\big) \Phi^1_\tau(\tau, X^*(\tau))
           \big],\quad P_1-a.s.
           \end{align}
           Note that $ \int_t^TE(||\Phi^i_s||_{C^{1,2}_1(\bar{Q}_T)}^2)ds< \infty$, and therefore $E(||\Phi^i_s||_{C^{1,2}_1(\bar{Q}_T)}^2)<\infty$ for a.e. $s\in[t,T]$ and $i\in\{1,2\}$. Therefore, without loss of generality,  we may assume $E(||\Phi^1_\tau||_{C^{1,2}_1(\bar{Q}_T)}^2)<\infty$ i.e. 
           \begin{align}
           \label{eq:fatou-just} E(||\phi||_{C^{1,2}_1(\bar{Q}_T)}^2)<\infty. 
           \end{align} Now use \eqref{eq:ito-p2-average}, and boundedness of the data to conclude 
           \begin{align}
                & \nonumber E^{P_2} \Big[ \frac{u(\tau+h,X^*(\tau+h))-u(\tau, X^*(\tau))}{h}\Big]\\\nonumber \le& C||\phi||_{C^{1,2}_1(\bar{Q}_T)} \frac{1}{h}\int_\tau^{\tau+h}(1+E^{P_2}|X^*(s)|)ds\\ 
                                                       \nonumber            \le &  C ||\phi||_{C^{1,2}_1(\bar{Q}_T)} \frac{1}{h}\int_\tau^{\tau+h}(1+|X^*(\tau)|+C^\prime \sqrt{\tau-s})ds\\
                                                       \label{eq:faous-2}     \le &          C ||\phi||_{C^{1,2}_1(\bar{Q}_T)}(1+|X^*(\tau)|+C^\prime) := \rho_\tau(\omega),                                       
           \end{align}where $h\le 1$ and we have used Lemma \ref{lem:A1}. Also, by Cauchy-Schwartz inequality and \eqref{eq:fatou-just}, we have $E(|\rho_\tau(\omega)|) < \infty$ as $X^*(\tau)$ has bounded second moments.

 By \eqref{eq:faous-2}, we can now apply Fatou's lemma and use \eqref{eq:est-final-1} to obtain
           \begin{align*}
            &\limsup_{h\downarrow 0} E\frac 1h\big[u(\tau+h, X^*(\tau+h))-u(\tau,X^*(\tau))\big]\\
            =& \limsup_{h\downarrow 0} E\frac 1h\Big[E\big[u(\tau+h, X^*(\tau+h))
            -u(\tau,X^*(\tau))\big]\big|{\mathcal{F}_{t,\tau}}\Big]\\
            =  & \limsup_{h\downarrow 0} E\Big[E^{P_2}\frac 1h\big[u(\tau+h, X^*(\tau+h))
            -u(\tau,X^*(\tau))\big]\Big]\\
            \le &  E\Big[\limsup_{h\downarrow 0}E^{P_2}\frac 1h\big[u(\tau+h, X^*(\tau+h))
            -u(\tau,X^*(\tau))\big]\Big]\\
            \le &E\Big[p^1(\tau)+\langle q^1(\tau), b^*(\tau)\rangle+\frac12\text{Tr}(\sigma^*{\sigma^*}^T(\tau)Q^1(\tau))
           \\\nonumber& \hspace{1cm}+ \mathcal{J}\big(\tau,X^*(\tau); Y^*(\tau), Z^*(\tau)\big) \Phi^1_\tau(\tau, X^*(\tau))
           \Big],
           \end{align*} for almost every $\tau\in [t,T]$.  To this end, we define $G(\tau)=E(\tau,X^*(\tau))$ and observe   
           \begin{align}
            \label{eq:est-final-2} &\limsup_{h\downarrow 0}\frac{G(\tau+h)-G(\tau)}{h} \\\nonumber\le& E\Big[p^1(\tau)+\langle q^1(\tau), b^*(\tau)\rangle+\frac12\text{Tr}(\sigma^*{\sigma^*}^T(\tau)Q^1(\tau))
           \\\nonumber& \hspace{1cm}+ \mathcal{J}\big(\tau,X^*(\tau); Y^*(\tau), Z^*(\tau)\big)
            \Phi^1_\tau(\tau, X^*(\tau))\Big],
          \end{align} for a.e. $\tau \in[t,T].$  We proceed similarly as in \eqref{eq:faous-2} and argue, for  all $h\ge 0$, that
          
          \begin{align}
            \nonumber  \frac{G(\tau+h)-G(\tau)}{h}&\le  E \Big[ \frac{\Phi_\tau^1(\tau+h,X^*(\tau+h))-\Phi_\tau^1(\tau, X^*(\tau))}{h}\Big]\\
            \nonumber  & \le C \frac{1}{h}\int_\tau^{\tau+h}E\big[  ||\Phi^1_\tau||_{C^{1,2}_1(\bar{Q}_T)}(1+|X^*(s)|)\big]ds\\
          \nonumber    & \le  C\Big [   \frac{1}{h}\int_\tau^{\tau+h}E\big( ||\Phi^1_\tau||_{C^{1,2}_1(\bar{Q}_T)}^2\big)ds  \Big]^{\frac{1}{2}}\Big [   \frac{1}{h}\int_\tau^{\tau+h}E\big(1+|X^*(s)|^2\big)   \Big]^{\frac{1}{2}}\\
           \label{faous-3}   & \le C E\big( ||\Phi^1_\tau||_{C^{1,2}_1(\bar{Q}_T)}^2\big)^{\frac{1}{2}} := \rho^\prime(\tau),
          \end{align} where we have used It\^{o}-L\`{e}vy formula,  Cauchy-Schwartz inequality, boundedness of the data and boundedness of second moment of $X^*(\cdot)$. By our assumptions, the map $\tau \mapsto E\big( ||\Phi^1_\tau||_{C^{1,2}(\bar{Q}_T)}^2\big)$ is Borel measurable and a member of $L^1([0,T])$. Therefore the map $\rho^\prime(\tau)$ is in $L^2([0,T])$. We now apply Cauchy-Schwartz inequality once more and conclude  $\rho^\prime(\tau)\in L^1([0,T])$.
          
           We now refer to the explanation in \cite{Gozzi:2010bd} and conclude, in view of \eqref{faous-3}, that the Lemma 5.2 of  \cite[section 5.2]{Yong:2005fg} is valid for the function $G(s)$. 
        We apply this lemma  to $G(s)$ and use \eqref{eq:est-final-2}
          along with \eqref{eq:diff-sup-integral} to obtain
         \begin{align*}
            &G(T)-G(t)\le -E\int_t^T f^*(s)ds\\
            \text{i.e.}\quad& E(g(X^*(T)))-u(t,x)\le -E\int_t^T f^*(s)ds\\
             \text{or}\quad& J(t,x; Y^*,Z^*)\le u(t,x).
         \end{align*} Since the Isaacs condition \eqref{eq:isac_cond} holds and the value of the game exists i.e. $V(t,x)=U(t,x)$ ,
          we use Theorem \ref{thm:visc_connection} and comparison principle from Theorem \ref{thm:wellposed}
          to conclude
               \begin{align*}
                  J(t,x; Y^*,Z^*)\le u(t,x)\le V(t,x)=U(t,x),
               \end{align*} which gives half of the requirement for \eqref{eq:finalconc}. The other half also follows
                after we apply similar machinery to
                $v(t,x)$ and the conditions \eqref{eq:sub-diff},\eqref{eq:sub-diff-integral} . 
            
                 \end{proof}
          \subsection*{Remark on optimality of $(Y^*, Z^*)$}  It is perhaps misleading to call the control pair $(Y^*,Z^*)$ in Theorem \ref{thm:verification}  as optimal merely on the basis of the equality \eqref{eq:finalconc}. If for all $Z\in N(t)$, there exist  $(p^Z, q^Z, Q^Z; \Phi^Z)\in L_{\mathcal{F}_{t,\cdot}}^2(t,T;\mathbb{R})\times
        L_{\mathcal{F}_{t,\cdot}}^2(t,T;\mathbb{R}^d)\times L_{\mathcal{F}_{t,\cdot}}^2(t,T;\mathbb{S}^d)\times
        L_{\mathcal{F}_{t,\cdot}}^2(t,T;C_1^{1,2}(Q_T))$ such that $\Phi^Z$  is progressively measurable and  \eqref{eq:diff-sup}, \eqref{eq:diff-sup-integral} hold for  $(Y^*,Z)$ , then the same proof works and one  can conclude $J(t,x;Y,Z^*) \le u(t,x) \le V(t,x) = U(t,x)$.   Also, for every $Y\in M(t)$, if conditions similar to \eqref{eq:sub-diff}-\eqref{eq:sub-diff-integral}  hold for the control pair $(Y, Z^*)$ then one can show that $J(t,x; Y,Z^*) \ge v(t,x) > U(t,x)=V(t,x)$. In such a scenario, the pair $(Y^*, Z^*)$ will be a saddle-point control pair satisfying $J(t,x;Y^*,Z)\le J(t,x; Y^*, Z^*)\le J(t,x;Y,Z^*)$  for all admissible control pair $(Y,Z)$.

\appendix
\section{Proof of some techincal Lemmas}
\begin{proof}[Proof of Lemma  \ref{lem:eqv-visc} ] Fix $(t,x) \in Q_T$ and $(p,q,Q)\in D_{t,x}^{1,2+}v$. Let $\varphi = \varphi[p,q,Q]$ (merely indicates the dependence on p,q and Q) such that
\[[\varphi,\varphi_t,D\varphi, D^2\varphi](t,x)= [v(t,x),p,q,Q]\]
so that, by Lemma \ref{lem:superdiff-charac}, $v-\varphi$ has strict global maximum at $(t,x)$ relative to the set of points $(s,y)$ such that $s\ge t$. Without loss of generality we may assume  $t=0$. Now it is easy to see that, for small $\mu$, the function $v-\varphi-\frac \mu s$  will have strict global maximum at $(t_\mu,x_\mu)$ and $(t_\mu, x_\mu)\rightarrow (0,x)$ as $\mu\rightarrow 0$. Therefore, from Definition \ref{defi:visc}, it follows that 
\[-\varphi_t(t_\mu,x_\mu)-F(t_\mu,x_\mu,D\varphi(t_\mu,x_\mu), D^2\varphi(t_\mu,x_\mu),\varphi(t_\mu,\cdot))\le -\frac{\mu}{t_\mu^2}\]
for $F= H^+$ or $H^-$. We get the desired result by passing to the limit $\mu\downarrow 0$ and invoking continuity of $H^+$ or $H^-$. This gives the `if' part, the proof of  `only if' trivially follows from Lemma \ref{lem:superdiff-charac}.
\end{proof}

\begin{proof}[Proof of Lemma \ref{lem:lip_holder}]
   The Lipschitz continuity and boundedness of $V$ and $U$ are straight forward consequences of uniform Lipschitz continuity and uniform boundedness of $J(t,x;Y,Z)$ in $x$. As a result, it is enough to prove [$a.)$], in which case the uniform boundedness is evident from the boundedness of the data. The Lipschitz continuity will follow if we can prove, for $x,y\in \mathbb{R}^d$,
   \begin{align}
   \label{eq:appendix-1} E^{P_t}\big(|X_{t,x}-X_{t,y}|\big) \le C|x-y|,
   \end{align} where $X_{t,x}$ and $X_{t,y}$ are solutions of \eqref{eq:SDG_dymcs} starting at $t$, respectively from the points $x$ and $y$ with the same control pair $\gamma = (Y,Z)$. We proceed as follows:

   Set $Z(s)= X_{t,x}(s)-X_{t,y}(s) $.  By applying It\^{o}-L\`{e}vy formula , we have 
   
   \begin{align*}
    E^{P_t} \big(|Z(s)|^2\big) =&|x-y|^2\\
       &\quad + E^{P_t}\Big\{\int_s^t\big[2Z(r).\bar{b}(r, X_{t,x}(r), X_{t,y}(r); \gamma(r))\\
       &\quad +\text{Tr}(\bar{\sigma}\bar{\sigma}^T)(r, X_{t,x}(r), X_{t,y}(r);\gamma(r))\\
       & \quad \int_{\mathbb{E}}|\bar{\eta}(r, X_{t,x}(r), X_{t,y}(r),w;\gamma(r))|^2\nu(dw)\big]dr\Big\},
   \end{align*} where 
   \begin{align*}
   \bar{b}(r, X_{t,x}(r), X_{t,y}(r)); \gamma(r)) &= b(r, X_{t,x}(r); \gamma(r))- b(r, X_{t,y}(r); \gamma(r)),\\
    \bar{\sigma}(r, X_{t,x}(r), X_{t,y}(r)); \gamma(r)) &= \sigma(r, X_{t,x}(r); \gamma(r))- \sigma(r, X_{t,y}(r); \gamma(r)),\\
       \bar{\eta}(r, X_{t,x}(r), X_{t,y}(r)), w; \gamma(r)) &= \eta(r, X_{t,x}(r); \gamma(r); w)- \eta(r, X_{t,y}(r); \gamma(r);w).
   \end{align*} We invoke the Lipschtiz continuity assumption on the data and obtain
   \begin{align*}
     E^{P_t} \big(|Z(s)|^2\big) \le &|x-y|^2+C \int_t^s E^{P_t}(|Z(r)|^2)dr,
   \end{align*}where we have used Fubini's theorem. Now we apply Gronwall's inequality to deduce
   \begin{align*}
     E^{P_t} \big(|Z(s)|^2\big)\le (1+\int_t^se^{Cr}dr)|x-y|^2.
   \end{align*} In other words, we have just derived 
   \begin{align*}
    E^{P_t}\big( |X_{t,x}(s)-X_{t,y}(s)|\big) \le C|x-y|^2,
   \end{align*} which implies \eqref{eq:appendix-1} after applying Cauchy-Schwartz inequality. 
   \end{proof}

   We close this Appendix with the following known estimate for stochastic differential equations.
   
  \begin{lem} \label{lem:A1}
  Let $X(\cdot)$ be the solution of \eqref{eq:SDG_dymcs} starting from a point $x$ at time $t$, corresponding to an admissible control pair $(Y,Z)$.  Then, it holds that 
  \begin{align}
  \label{A3.1}E^{t,x}\big(|X(\tau)-x|\big)\le C\sqrt{\tau-t},
  \end{align}where $C$ is  a constant depending on the data.
  \end{lem} 
  
 \begin{proof}
    Without loss of generality we may assume  $|\tau-t|\le 1$. We have 
    \begin{align*}
    X(\tau) -x =& \int_t^\tau b(s, X(s); Y(s), Z(s))ds+\int_t^\tau \sigma(s,X(s);Y(s),
Z(s))dW(s)\\&\qquad+\int_t^\tau\int_{E}\eta(s, X(s^-);Y(s),
Z(s);w)\tilde{N}(ds,dw).
 \end{align*} Therefore, by It\^{o}-L\`{e}vy isometry, we conclude 
 \begin{align*}
 &E| X(\tau) -x|^2\\ \le& 3E\big [  \int_t^\tau b(s, X(s); Y(s), Z(s))ds\big]^2 +3\int_t^\tau E| \sigma(s,X(s);Y(s),Z(s))|^2ds
Z(s))ds\\
&+3\int_t^\tau\int_{E}E|\eta(s, X(s);Y(s),
Z(s);w)|^2\nu(dw)ds\\
\le &K(|\tau-t|+|\tau-t|^2),
 \end{align*} where we have used the bounded ness assumption on the data. Sine $|\tau-t|\le 1$, \eqref{A3.1} follows trivially from above estimate after applying Cauchy-Schwartz inequality.

  \end{proof}

\end{document}